\documentclass{article}

\usepackage{a4wide}
\usepackage{cite}
\usepackage[utf8]{inputenc}
\usepackage[T1]{fontenc}

\usepackage[colorlinks=true, pdfstartview=FitV, linkcolor=blue,citecolor=blue, urlcolor=blue]{hyperref}

\usepackage[french,english]{babel}
\usepackage{amsmath}
\usepackage{amscd}
\usepackage{amssymb}
\usepackage{stmaryrd}
\usepackage{amsrefs}
\usepackage{amsthm}
\usepackage{cases}
\usepackage{framed}
\usepackage{graphicx}
\usepackage{latexsym}
\usepackage{color}
\usepackage{array,multirow,makecell}
\usepackage[capitalise]{cleveref}

\usepackage{enumerate}
\usepackage{enumitem}  

\usepackage{bbm}

\setcellgapes{1pt}
\makegapedcells

\numberwithin{equation}{section}

\theoremstyle{plain} 
\newtheorem{thm}{Theorem}[section]
\newtheorem{prop}[thm]{Proposition}
\newtheorem{lem}[thm]{Lemma}

\theoremstyle{definition}

\theoremstyle{remark} 
\newtheorem{rmk}[thm]{Remark}

\newcommand{\R}{\mathbb{R}}

\newcommand{\N}{\mathcal{N}}
\renewcommand{\H}{\mathcal{H}}

\newcommand{\tX}{\widetilde{X}}
\newcommand{\tU}{\widetilde{U}}
\newcommand{\tP}{\widetilde{P}}
\newcommand{\tK}{\widetilde{K}}

\newcommand{\bx}{\mathbf{\rm x}}
\newcommand{\tx}{\mathbf{\rm \tilde x}}

\newcommand{\by}{\mathbf{\rm y}}
\newcommand{\ty}{\mathbf{\rm \tilde y}}
\newcommand{\bu}{\mathbf{\rm u}}

\newcommand{\bw}{\mathbf{\rm w}}

\newcommand{\tLambda}{\widetilde \Lambda}

\title{Microscopic derivation of the one-dimensional\\ constrained Euler equations}
\author{Charlotte Perrin\footnote{Aix Marseille Univ, CNRS, I2M, Marseille, France; charlotte.perrin@cnrs.fr}}

\begin{document}

\maketitle

\begin{footnotesize}
\centerline{\bf Abstract}
We provide a new existence result for weak solutions to the one-dimensional Euler equations with a maximal density constraint, corresponding to a unilateral constraint on the density. 
Such models arise in the description of congestion phenomena in compressible flows. 
Our approach is based on a microscopic approximation by a system of $N$ solid particles of identical radius $r$, with $2r = 1/N$. 
The particles move freely until collision, after which perfectly inelastic interactions are imposed, so that colliding particles stick together. 
At this level, the non-overlapping condition is encoded through Signorini-type constraints from contact mechanics. 
Passing to the limit as $N \to +\infty$, we rigorously establish the connection between these microscopic Signorini conditions and the macroscopic unilateral constraint on the density, together with the associated sign condition on the congestion pressure. 
The analysis is carried out in a Lagrangian framework, which is natural at the microscopic level and relies at the macroscopic level on the monotone rearrangement associated with the density. 
A key ingredient of our result is a monotonicity property of the congested region, which allows us to reduce the dynamics to a first-order evolution in time.


\end{footnotesize}

\bigskip

\section{Introduction}

In this paper, we investigate the one-dimensional pressureless Euler system under a maximal density constraint:
\begin{subnumcases}{\label{eq:macro-euler}}
\partial_t \rho + \partial_x(\rho u) = 0,\\
\partial_t(\rho u) + \partial_x(\rho u^2) + \partial_x p = 0,\\
0 \leq \rho \leq 1, \quad (1-\rho)\,p = 0, \quad p \geq 0.
\end{subnumcases}
Here, $\rho$ denotes the density of the fluid and $u$ its velocity. The pressure $p$ is associated with the constraint $\rho \leq 1$ and can be interpreted as a Lagrange multiplier enforcing incompressibility of the velocity field inside congested regions where $\rho = 1$.

The system~\eqref{eq:macro-euler} was first introduced by Bouchut \emph{et al.}~\cite{bouchut2000} as an asymptotic model for liquid--gas mixtures, in the regime where the gas density is negligible compared to that of the liquid. Since then, it has been widely studied as a simplified model for congestion phenomena (see~\cite{perrin2026} and references therein). Such effects arise in a variety of applications, ranging from collective motion (crowd dynamics~\cite{maury2011}, traffic flow) to geophysical flows (granular media and suspensions~\cite{lefebvre2011}), as well as free-surface and partially free-surface flows~\cite{lannes2017,godlewski2018}.

The present work is mainly motivated by applications to granular media and collective dynamics. In this context, the macroscopic density constraint $\rho \leq 1$ can be understood as the continuum counterpart of a microscopic packing constraint: at the particle level, solid grains (or agents) are subject to a non-overlapping condition, which prevents excessive compression. Combined with contact dynamics, this leads to the activation of a unilateral constraint, often referred to as a Signorini condition.

The main goal of this paper is to rigorously derive the macroscopic model~\eqref{eq:macro-euler} from a microscopic particle system, and to establish a precise link between discrete dynamics with non-overlapping constraints and the corresponding Eulerian description.

\bigskip
More precisely, the natural microscopic model to consider consists of $N > 1$ solid particles interacting only through inelastic collisions (once particles collides, they stick). The dynamics can be written in the following form:
\begin{subnumcases}{\label{eq:micro-1st}}
u_i + N\big(\lambda_i - \lambda_{i-1}\big) = u^0_i,  \quad \text{for } i= 1,\dots, N,\label{eq:micro-1st-ID}\\
x_{i+1} - x_i \geq 2r, \quad \big(x_{i+1} - x_i - 2r\big)\,\lambda_i = 0, \quad \lambda_i \geq 0,\label{eq:micro-1st-compl}
\end{subnumcases}
where $u^0_i$ denotes the initial velocity of particle $i$, and $\lambda_i$ represents the (rescaled) contact force between particles $i$ and $i+1$. The conditions~\eqref{eq:micro-1st-compl}, known as complementarity conditions, correspond to the so-called Signorini conditions. At the boundaries, one naturally imposes
\[
\lambda_0 = \lambda_N = 0.
\]
At the microscopic level, the multipliers $\lambda_i$ encode the reaction forces associated with the non-overlapping constraint, and the dynamics is of first order in time.
In order to relate this system to the macroscopic Euler model~\eqref{eq:macro-euler}, which is of second order in time, one is led to formally differentiate the system and obtain
\begin{subnumcases}{\label{eq:micro-2nd}}
\dot{u}_i + N\big(p_i - p_{i-1}\big) = 0,  \quad \text{for } i= 1,\dots, N,\label{eq:micro-2nd-ID}\\
x_{i+1} - x_i \geq 2r, \quad \big(x_{i+1} - x_i - 2r\big)\,p_i = 0, \quad p_i \geq 0,\label{eq:micro-2nd-compl}
\end{subnumcases}
where $p_i = \dot{\lambda}_i$.
However, it is not immediate to ensure that the complementarity condition~\eqref{eq:micro-2nd-compl} is preserved at the level of the time derivatives $p_i$. 

The first objective of this paper is therefore to establish a rigorous equivalence between the first-order~\eqref{eq:micro-1st} and second-order~\eqref{eq:micro-2nd} microscopic formulations. 
The second and main objective is to justify the mean-field limit $N \to +\infty$ and to derive a solution to the constrained Euler system~\eqref{eq:macro-euler}.
This provides a rigorous bridge between particle dynamics with unilateral constraints and the macroscopic Euler system with congestion.

\bigskip

\noindent{\bf A brief state-of-the-art on Equations~\eqref{eq:macro-euler}.}
In Berthelin's work~\cite{berthelin2002}, the existence of weak solutions to the constrained Euler equations~\eqref{eq:macro-euler} is obtained by passing to the limit in a class of explicit solutions known as sticky blocks. This construction allows one to capture the formation of congested regions and leads to solutions satisfying several important properties, including a maximum principle on the velocity, an Oleinik-type one-sided Lipschitz estimate.
A major difficulty in this framework is that the pressure $p$ is only a measure, while the density $\rho$ belongs to $L^\infty$, so that the product $\rho p$ is not well defined in general. Formally, the constraint $(1-\rho)p=0$ amounts to $\rho p = p$, but this relation cannot be interpreted directly at the weak level.
To overcome this issue, Berthelin introduces an alternative formulation of the constraint based on nonlinear quantities. More precisely, he considers two auxiliary potentials $Q$ and $M$ defined through
\[
\partial_x Q = \rho u, \qquad \partial_t Q = -(\rho u^2 + p),
\]
\[
\partial_x M = \rho, \qquad \partial_t M = -\rho u,
\]
and defines two quantities
\[
R_1 = -\partial_t(\rho Q) - \partial_x(\rho u Q), \qquad
R_2 = \partial_t(\rho u M) + \partial_x\big((\rho u^2 + p)M\big).
\]
For smooth solutions, a direct computation shows that
\[
R_1 = R_2 = \rho p,
\]
so that imposing the condition
\[
R_1 = R_2 = p
\]
in the sense of distributions provides a weak formulation of the constraint $\rho p = p$.
This reformulation makes it possible to pass to the limit in the approximate solutions despite the low regularity of $p$, at the cost of introducing nonlocal quantities and a more involved structure of the equations.

From a modeling and numerical viewpoint, the pressureless Euler system with congestion also raises fundamental issues related to non-uniqueness and the lack of a canonical collision law. This aspect is discussed in detail by Maury and Preux in~\cite{maury2017}, where the authors emphasize that the constrained Euler admits infinitely many weak solutions, including non-physical ones. 
In particular, even imposing natural conditions such as energy dissipation is not sufficient to restore uniqueness. 
As illustrated in~\cite{maury2017}, one can construct multiple energy-decreasing solutions corresponding to different collision behaviors (e.g. sticking versus partial rebound with a restitution coefficient). 
This highlights the fact that the macroscopic system lacks an intrinsic selection principle, which should encode the underlying microscopic interaction mechanism.
This observation is closely related to the absence of a well-defined collision law at the macroscopic level. 
In the microscopic granular model, interactions between particles are governed by unilateral constraints and impact laws (possibly involving restitution coefficients), which determine uniquely the post-collisional velocities. 
In contrast, the macroscopic system~\eqref{eq:macro-euler} only enforces the congestion constraint through the pressure, without prescribing how momentum is redistributed during collisions.\\
To address this issue, Maury and Preux propose a time-splitting scheme combining a transport step with a projection onto the admissible set of densities, inspired by optimal transport methods and crowd motion models. 
While this approach provides a consistent numerical framework and captures relevant qualitative features of congested flows, the convergence towards a well-posed macroscopic dynamics remains largely open~\cite{preux2016}.\\
These considerations further motivate the present work, where the macroscopic dynamics is derived from a microscopic model with inelastic collisions. In this setting, the collision law is encoded at the particle level and is preserved in the limit, leading to a more intrinsic selection of admissible solutions.

\medskip
\noindent{\bf Lagrangian standpoint.}
An alternative and fruitful point of view to study fluid dynamics in one space dimension is provided by Lagrangian mass coordinates.
In this framework, the density $\rho(t,x)$ is represented through a monotone map $X(t,\cdot)$, defined as the pseudo-inverse of the cumulative distribution function:
\[
X(t,w) := \inf \left\{ x \in \R \;:\; \int_{-\infty}^x \rho(t,y)\,dy \geq w \right\}, \quad w \in (0,1),
\]
so that
\[
\rho(t,\cdot) = X(t,\cdot)_{\#} \, \mathcal{L}^1_{(0,1)}.
\]
This Lagrangian formulation has been extensively developed in the analysis of pressureless-type systems, where it provides a powerful framework to construct and characterize weak solutions. 
In particular, it plays a central role in the study of pressureless gas dynamics, allowing one to encode the dynamics through monotone transport maps and to exploit the underlying convex structure of the problem.
Foundational contributions in this direction include the work of Natile and Savaré~\cite{natile2009}, where the Lagrangian formulation is used to describe the evolution of sticky particle systems and to establish the convergence towards macroscopic solutions. 
This approach was further developed by Cavalletti et al~\cite{cavalletti2015}, who construct global weak solutions directly in Lagrangian coordinates by means of projection techniques onto convex sets of monotone maps.
The Lagrangian point of view has also been extended to more complex systems, such as Euler--Poisson equations, in the work of Brenier et al.~\cite{brenier2013}, where similar ideas are used to capture the interaction between transport and self-consistent forces.
A key feature of this framework is that it establishes a direct link between macroscopic weak solutions and the underlying microscopic dynamics of sticky particles, providing both a constructive approach and a natural selection principle for admissible solutions.

The Lagrangian framework can also be adapted to incorporate a congestion constraint. 
Indeed, in these variables, the congestion constraint $\rho \leq 1$ translates into the simple monotonicity condition
\[
\partial_w X(t,w) \geq 1.
\]
In this direction, it has been successfully applied in~\cite{perrin2018} to macroscopic models of granular flows with memory effects, where the congestion constraint is coupled with an additional adhesion variable. 
In~\cite{perrin2018}, the authors consider a one-dimensional granular flow model involving an adhesion potential $\gamma$, which encodes the history of congestion and satisfies a transport-reaction equation coupled with the momentum balance.
In this context, the Lagrangian description allows one to construct global weak solutions by exploiting the monotonicity of the transport map and interpreting the dynamics as a projection onto a set of admissible configurations, in the spirit of pressureless gas dynamics.\\
In~\cite{kim2024}, Kim et al. consider a first-order in time microscopic dynamics of particles subject to a non-overlapping constraint, where the evolution of the associated Lagrangian map is governed by
\[
\partial_t X = -\phi'(X) - \partial_w \Lambda,
\]
together with the complementary conditions
\[
\Lambda \geq 0, \quad \partial_w X \geq 1, \quad \Lambda(1-\partial_w X)=0.
\]
This equation can be interpreted as the projection of the spontaneous velocity field $-\phi'(X)$ onto the set of admissible velocities preserving the congestion constraint. 
The above system is equivalent, in Eulerian variables, to a continuity equation for the density $\rho$ coupled with a pressure $p$ enforcing the constraint $\rho \leq 1$, namely
\[
\partial_t \rho - \partial_x\big(\rho(\phi' + \partial_x p)\big) = 0,
\qquad
0 \leq \rho \leq 1,  \quad p(1-\rho)=0, \quad p \geq 0.
\]

In contrast, the present work addresses a second-order in time dynamics, consistent with the Euler equations~\eqref{eq:macro-euler}.
While the Lagrangian framework remains central, this setting raises additional difficulties at the level of time regularity. 
Indeed, the microscopic model is naturally formulated as a first-order differential inclusion in time, and passing to a second-order formulation formally requires differentiating the Lagrange multipliers associated with the constraint. 
However, these multipliers typically have low regularity in time, making this step non-trivial and requiring a careful analysis.

\bigskip
\noindent{\bf Main result.} 
We can now state the main result of the present paper.

\begin{thm}[Main result]\label{thm:main}
	Let $(\rho^0,u^0)$ be an initial Eulerian datum such that
	\begin{align}
	\rho^0\in \mathcal P_2(\mathbb R), \quad \rho^0\ll \mathcal L^1, \quad \rho^0 \leq 1,
	\quad \operatorname{supp}\rho^0 \Subset \R,\\
	u^0 \in L^2(\R,\rho^0) \cap L^\infty(\R,\rho^0), \quad \partial_x u^0 = 0 \quad \text{in} ~\{\rho^0 = 1\}. \label{hyp:u0}
	\end{align}
	Let $(X^0,U^0)$ be such that
	\begin{equation}
	\rho^0 = X^0_{\#}\mathcal L^1_{(0,1)},
	\quad 
	u^0 = U^0 \circ X^0.
	\end{equation}
	
	\noindent\textbf{(i) Discretization of the initial datum.}
	Let $N> 1$, and consider $N$ particles of identical radius $r>0$ such that
	\[
	2r = \frac{1}{N}.
	\]
	There exists a sequence of positions and velocities $(\bx^0_N,\bu^0_N) = \big((x_i^0)_{i=1,\dots,N}, (u_i^0)_{i=1,\dots,N}\big)$ such that
	\begin{equation}
	x_{i+1}^0 - x_i^0 \geq 2r,
	\qquad
	u_{i+1}^0 = u_i^0 \ ~\text{whenever}\ x_{i+1}^0 - x_i^0 = 2r,
	\end{equation}
	and if $(X_0^N,U_0^N)$ denote the associated piecewise constant interpolations, then
	\begin{equation}
	X^0_N \to X^0,
	\qquad
	U^0_N \to U^0
	\quad \text{strongly in } L^2(0,1).
	\end{equation}
	
	\medskip
	
	\noindent\textbf{(ii) Microscopic dynamics.}
	There exists a unique trajectory $\bx^N \in \mathrm{Lip}([0,T];\mathbb R^N)$ solving~\eqref{eq:micro-1st}.
	Let $X_N$ be a suitable approximation of $\bx^N$. One can associate a velocity field $U_N$ and a pressure field $P_N$ such that $(X_N,U_N,P_N)$ is solution to the PDE system
	\begin{equation}{\label{eq:PDE-Lagr-N}}
	\begin{cases}
	\partial_t X_N = U_N,\\
	\partial_t U_N + \partial_w P_N = 0,
	\end{cases}
	\end{equation}
	in the sense of distributions, together with the complementary relation
	\begin{equation}\label{eq:PDE-Lagr-N-compl}
	\partial_w X_N \geq 1,
	\qquad	(\partial_w X_N - 1)\,P_N = 0,
	\qquad	P_N \geq 0,
	\end{equation}
	and the Oleinik-type inequality
	\begin{equation}\label{eq:Oleinik-tUN-intro}
	\partial_w U_N < \dfrac{\partial_w X_N}{t} \cdot
	\end{equation}

	\medskip
	
	\noindent\textbf{(iii) Convergence towards a macroscopic solution.}
	Up to extraction of a subsequence, $(X_N,U_N,P_N)$ converges in a weak sense towards $(X,U,P)$ which defines a weak (distributional) Eulerian solution $(\rho, \rho u, p)$ of the constrained Euler equation~\eqref{eq:macro-euler} satisfying the Oleinik inequality $ \partial_x u < 1/t$ in $\mathcal{D}'$.
\end{thm}

\bigskip
\begin{rmk}
For the sake of clarity and concision, the statement of Theorem~\ref{thm:main} involves several simplifications and mild abuses of notation.\\
First, the Lagrangian system~\eqref{eq:PDE-Lagr-N}-\eqref{eq:PDE-Lagr-N-compl} is written in a compact form, although it actually involves different interpolations of the discrete variables $\bx^N$ (see Lemma~\ref{lem:discPDE2} for a precise construction).\\
Second, the convergence of $(X_N,U_N,P_N)$ is only stated in a weak sense. 
In practice, the passage to the limit in the nonlinear constraint~\eqref{eq:PDE-Lagr-N-compl} requires strong convergence of at least one component, which is obtained through suitable compactness estimates.	
We refer to Proposition~\ref{prop:cvg1} and Proposition~\ref{prop:2ndorder-LagrPDE} for the precise functional framework and convergence results.
\end{rmk}

\bigskip
\begin{rmk}
\begin{itemize} 
	\item A first difficulty lies in obtaining suitable estimates on the multipliers. 
	The standard Karush-Kuhn-Tucker (KKT) theory does not directly provide sign or time-regularity information on the derivatives of the multipliers $\lambda_i$ arising in the Signorini conditions. 
	In particular, to justify the second-order complementarity relation (positivity and bounds on the pressure), the projection formula and its semigroup property play a crucial role. 
	To pass to the limit $N \to +\infty$ we derive an important bound on the pressure that is uniform with respect to $N$ (see Proposition~\ref{prop:pressure-bound}). 
	\item The Oleinik condition naturally appears at the microscopic level as a consequence of the sticky particle dynamics. 
	It provides a control on the velocity gradient, see~\eqref{pwUN-L1}, which is essential to pass to the limit in the nonlinear complementarity relation (see the proof of Proposition~\ref{prop:2ndorder-LagrPDE}). 
\end{itemize}
	
\end{rmk}

\bigskip
\begin{rmk}
A major difficulty in Berthelin's approach is that the pressure $p$ is only a measure, while the density $\rho$ belongs to $L^\infty$, so that the product $\rho p$ is not well defined in general. 
This reflects the lack of a canonical representative of $\rho$ on the support of $p$.
In contrast, the use of Lagrangian coordinates provides additional structure.
The density is obtained as the push-forward of the Lebesgue measure through a monotone map, while the pressure is supported on the congested region where $\rho=1$. 
This allows one to give a direct meaning to the relation $\rho p = p$ without resorting to an alternative formulation.\\
Another important difference concerns the absence of a macroscopic shock selection principle in Eulerian formulations (see, for instance Maury and Preux \cite{maury2017} previously discussed).
Here, the solution is constructed from a Lagrangian dynamics satisfying a projection principle on the velocity, which encodes the inelastic collision mechanism. 
This provides a natural selection criterion for admissible solutions. 
Actually, when our microscopic dynamics is reformulated as a first-order differential inclusion:
\[
\dot \bx(t) + \N_{\bx(t)} K^N \ni \bu^0,
\]
we fit into the framework of sweeping processes in the sense of Moreau~\cite{moreau1977}, $\N_{\bx(t)} K^N$ being the normal cone to the set of admissible configurations $K^N$ (i.e. satisfying the non-overlapping constraint) at $\bx(t)$, see~\eqref{df:normal-cone}.
This viewpoint provides a natural explanation for the projection structure and its role as a selection principle.\\
Issues related to uniqueness are discussed separately in Appendix~\ref{app:uniqueness}.
\end{rmk}

\bigskip
\noindent{\bf Outline.}
The paper is organized as follows.
In Section~\ref{sec:micro}, we introduce the microscopic dynamics and the associated discrete functional framework. We revisit the Karush--Kuhn--Tucker theory in this setting in order to characterize the first-order dynamics, and derive suitable estimates on the multipliers. We then extend the analysis to obtain a second-order formulation in time.
In Section~\ref{sec:limitN}, we perform the passage to the limit as $N \to +\infty$ at both the first- and second-order levels. This allows us to construct a macroscopic Lagrangian solution and to recover the Eulerian system~\eqref{eq:macro-euler}.
Finally, the appendix gathers several complementary results. 
It includes a detailed presentation of the Karush--Kuhn--Tucker framework classically used in the discrete setting, technical lemmas required in the convergence analysis, and a discussion on the issue of uniqueness for the macroscopic systems.


\section{Microscopic dynamics}\label{sec:micro}

\subsection{Discretisation of the macroscopic initial datum}

Let $(\rho^0,u^0)$ be an initial datum satisfying the assumptions of Theorem~\ref{thm:main}. 
In particular, $\rho^0$ is a probability measure with compact support, absolutely continuous with respect to the Lebesgue measure, and such that $\rho^0 \leq 1$.
We first construct a discrete approximation of $\rho^0$ by means of a quantile sampling. 
Let $X^0 : (0,1) \to \mathbb R$ be the monotone transport map such that 
\[ 
\rho^0 = X^0_{\#}\mathcal L^1_{(0,1)}. 
\] 
For a given integer $N \geq 2$, we define the discrete positions $\bx^0 = \bx^0_N = (x_i^0)_{i=1,\dots,N}$ by 
\[ 
x_i^0 := X^0\!\left(\frac{i}{N}\right), \qquad i=1,\dots,N. 
\]
By monotonicity of $X^0$, the sequence $(x_i^0)$ is nondecreasing. 
By construction, one has
\[
\int_{x_i^0}^{x_{i+1}^0} \rho^0(x)\,dx = \frac{1}{N}.
\]
Since $\rho^0 \leq 1$, it follows that
\[
x_{i+1}^0 - x_i^0 \geq \dfrac{1}{N},
\]
which shows that the particles of radius $2r = 1/N$ are initially non-overlapping.

We then define the Lagrangian velocity $U^0$ such that
\[
u^0 = U^0 \circ X^0. 
\]
and the discrete velocities $\bu_N^0 = (u_i^0)_{i=1,\dots,N}$ by sampling the Lagrangian velocity $U^0$: 
\[ 
u_i^0 := U^0\!\left(\frac{i}{N}\right), \qquad i=1,\dots,N.
\]
Let us now check that if two particles are in contact then they have the same velocity.
From the Eulerian incompressibility condition~\eqref{hyp:u0} satisfied by the initial velocity $u^0$, we derive
\[
\partial_w U^0(w) = \partial_x u^0(X^0(w)) \partial_w X^0(w) = 0 \quad \text{on} \quad \{\partial_w X^0 = 1\}.
\]
Since $\partial_w X^0 \geq 1$, we always have $X^0\left(\dfrac{i+1}{N}\right) - X^0\left(\dfrac{i}{N}\right) \geq \dfrac{1}{N}$.
So, if two particles are in contact, the equality imposes
\[
\partial_w X^0 = 1 \quad \text{a.e. in} ~\left(\dfrac{i}{N},\dfrac{i+1}{N}\right).
\] 
Therefore, by our incompressibility assumption, $\partial_w U^0 = 0 \quad \text{a.e. in} ~\left(\dfrac{i}{N},\dfrac{i+1}{N}\right)$, and finally deduce
\[
u^0_{i+1} = U^0\left(\dfrac{i+1}{N}\right) = U^0\left(\dfrac{i}{N}\right) = u^0_i.
\]

\subsection{Discrete configuration space}

The microscopic system~\eqref{eq:micro-2nd} can be interpreted as a second-order differential inclusion in the configuration space of particle positions.
Let
\[
\bx=(x_1,\dots,x_N)\in\R^N,
\qquad
\bu=(u_1,\dots,u_N)\in\R^N
\]
denote respectively the positions and velocities of the particles.
Throughout the paper we equip $\R^N$ with the rescaled Euclidian norm $\|\cdot\|$: 
\[
\|\bx\|^2 := \frac{1}{N} \sum_{i=1}^N |x_i|^2,
\] 
which will be consistent with the $L^2(0,1)$ norm of the associated interpolation functions defined in the next section.
Nevertheless, we keep the notation $\langle \cdot,\cdot \rangle$ for the standard (non-rescaled) Euclidean inner product (which is used only in this section, where $N$ is fixed)
\[
\langle \bx,\by\rangle := \sum_{i=1}^N x_i y_i,
\]
so that
\[
\|\bx\|^2 = \frac{1}{N}\langle \bx,\bx \rangle.
\]
After introducing the configuration space, we now describe the geometric constraints induced by the non-overlapping condition between particles.
The admissible configurations are given by the convex set
\begin{equation}\label{df:KN}
K^N
=
\{\bx\in\R^N;\; x_{i+1}-x_i\geq 2r,\quad i=1,\dots,N-1\},
\end{equation}
This constraint expresses the fact that particles of radius  $r > 0$ cannot overlap.
This set is closed and convex.

\medskip

\begin{rmk}[Translation of the configuration space]
	In the following, it will be sometimes convenient to remove the minimal spacing between particles by introducing the translated configuration
	\begin{equation}\label{df:translat-x}
	\tx = (x_1, x_2-2r, \dots, x_N -2r(N-1)) .
	\end{equation}
	In these new coordinates the admissible configurations become
	\begin{equation}\label{df:tKN}
	\tK^N
	=
	\{\tx\in\R^N;\; \tilde x_{i+1}- \tilde x_i \geq 0,\quad i=1,\dots,N-1\},
	\end{equation}
	i.e. $\tK^N$ is the cone of monotone (non-decreasing) vectors.
\end{rmk}

\medskip

We denote by $I_{K^N}$ denotes the indicator function of $K^N$, defined by
\begin{equation*}
I_{K^N}(\bx) = \begin{cases}
0 \quad &\text{if}~ \bx \in K^N,\\
+ \infty \quad &\text{otherwise}.
\end{cases}
\end{equation*}
Its subdifferential coincides with the normal cone
\begin{equation}\label{df:normal-cone}
\partial I_{K^N} = \N_{\bx} K^N, \quad \text{where} \quad \N_{\bx} K^N = \{ \lambda \in \R^N; \langle \lambda, \by - \bx \rangle \leq 0, ~ \forall \ \by \in K^N\}.
\end{equation}
This cone represents the reaction forces enforcing the non-overlapping constraint.\\
The non-overlapping constraint also restrict admissible velocities. 
Whenever two particles are in contact, their relative velocity cannot be negative, which leads to the cone
\begin{equation}\label{def:velcone}
\mathcal C_\bx^N
=
\{\bu\in\R^N;\;
u_{j+1}\ge u_j
\ \text{whenever } x_{j+1}=x_j+2r\}.
\end{equation}
These constraints naturally lead to a formulation of the microscopic dynamics as a differential inclusion in the configuration space.

The admissible velocities can be interpreted geometrically through the tangent cone. 
For any configuration $\bx \in K^N$, it is not difficult to check that the admissible velocities coincide with the tangent cone to $K_N$ at $\bx$
\[
\mathcal C_\bx^N
= \mathcal{T}_\bx K^N := \overline{\{\theta(\by - \bx); \ \by \in K^N, \ \theta \geq 0\}},
\]
and the normal cone is the polar cone of the tangent cone:
\[
\N_{\bx} K^N = (\mathcal{T}_\bx K^N)^\circ.
\]

\medskip
With these notations, the microscopic dynamics of the particles can be formulated as a differential inclusion:
\begin{equation}\label{eq:diff-order2}
\begin{cases}
\dot\bx(t)=\bu(t),\\
\dot\bu(t)\in -\partial I_{K^N}(\bx(t)),
\end{cases}
\end{equation}
supplemented with initial data
\[
\bx(0)=\bx^0\in K^N,
\qquad
\bu(0)=\bu^0\in\mathcal C_{\bx^0}^N .
\]
The second relation of~\eqref{eq:diff-order2} expresses that the acceleration is given by the reaction forces associated with the constraint $\tx \in K^N$.

\subsection{Sticky dynamics and reduction to first order}

Collisions in the system are perfectly inelastic: whenever two particles collide they stick together and subsequently move with the same velocity. As a consequence, clusters of particles can only grow in time and never break apart.
This stickiness property allows us to reformulate the microscopic dynamics as a first-order differential inclusion.

\medskip
For a given configuration $\bx \in K^N$, we denote by
\begin{equation}\label{df:Omega-micro}
\Omega_{\bx}^N
=
\{j\in\llbracket1,N-1\rrbracket;\; x_{j+1} = x_j + 2r\},
\end{equation}
the set of active constraints, that is the set of particle pairs that are in contact.
Particles belonging to the same contact cluster must share the same velocity.
We therefore introduce the linear subspace
\begin{equation}\label{df:H-micro}
\H_{\bx}^N
= \{\bu\in\R^N;\; u_{j+1}=u_j \ \text{for } j\in\Omega_{\bx}^N\}.
\end{equation}
The stickiness of the dynamics implies that once two particles collide they remain in contact afterwards. 
In particular the set of active constraints is non-decreasing in time.
This property translates into a monotonicity property of the subdifferentials of the indicator function of $K^N$. 
More precisely, if $t \mapsto x(t)$ is a solution of the microscopic dynamics, then for all $s \leq t$,
\begin{equation}
\partial I_{K_N}(\bx(s)) \subset \partial I_{K^N}(\bx(t)),
\end{equation}
or, equivalently,
\[
\N_{\bx(s)}K^N \subset \N_{\bx(t)}K^N
\qquad\text{for all } s\le t .
\]
This reflects the fact that new contact constraints may appear as time evolves, but existing ones cannot disappear.

\bigskip

Integrating the second-order differential inclusion in time therefore yields the following first-order formulation
\begin{equation}\label{eq:micro-f0-lagr}
\begin{cases}
\dot\bx(t) \in \bu^0 - \partial I_{K^N}(\bx(t)),\\
\bx(0)=\bx^0.
\end{cases}
\end{equation}

\medskip
The dynamics can be characterized explicitly in terms of the projection onto the convex set $K^N$.

\medskip

\begin{prop}\label{prop:proj-micro}
	Let $x_0\in K^N$ and $u_0\in C_{\bx_0}^N$. Then there exists a unique solution $\bx$ to~\eqref{eq:micro-f0-lagr}
	It is given by the projection formula
	\begin{equation}\label{eq:proj-micro-x}
	\bx(t) = \mathbb{P}_{K^N}(\bx_0 + t\,\bu_0), \qquad t \geq 0,
	\end{equation}
	where $\mathbb{P}_{K^N}$ denotes the metric projection onto $K^N$.
	
	Moreover $\bx$ is Lipschitz continuous and therefore differentiable for almost every $t \geq 0$. 
	Denoting $\bu(t):=\dot \bx(t)$, one has
	\begin{equation}
	\bu(t)\in \H_{\bx(t)}^N	\qquad\text{for a.e. } t \geq 0,
	\end{equation}
	and more precisely
	\begin{equation}\label{eq:proj-micro-u}
	\bu(t) = \mathbb{P}_{\H_{\bx(t)}^N}(\bu_0),
	\end{equation}
	where $\mathbb{P}_{\H_{\bx(t)}^N}$ denotes the orthogonal projection onto the subspace $\H_{\bx(t)}^N$.\\
    As a consequence, there exist non-negative multipliers $(\lambda_i)_{i=0,\dots,N}$ such that
    \begin{equation}\label{eq:dyn-micro}
    \bu(t) = \bu_0 - N\sum_{i= 1}^{N-1} \lambda_i(t) \big(e_{i} - e_{i+1}\big), \quad \lambda_0(t) = \lambda_N(t) = 0,
    \end{equation} 
    and
    \begin{equation}
    	\lambda_i(t) \big( x_{i+1}(t) - x_i(t) -2r \big)=0.
    \end{equation}
\end{prop}

\bigskip
\begin{rmk}
	\begin{itemize}
		\item The existence result itself is classical. In principle it could be obtained directly from the Karush–Kuhn–Tucker theory (see Appendix~\ref{app:KKT}). However, Proposition~\ref{prop:proj-micro} provides an explicit characterization of the solution in terms of projections and cluster velocities, which will play a crucial role in the analysis of the limit system.
		In particular, we will need these elements of proof to derive properties on the time derivatives of the multipliers, see Proposition~\ref{prop:micro-pressure}
		\item Normalization by $N$ of the multipliers $\lambda_i$ is a priori arbitrary, it is a choice made to pass to the limit $N \to +\infty$ in the next section.
	\end{itemize}
		
\end{rmk}

\bigskip
Let us first recall the following lemma concerning the variational characterization of the projection whose proof is classical (see for instance~\cite{brenier2013}, Section 2.2) and therefore omitted here.
\begin{lem}\label{lem:proj-characterization}
	Let \(C\subset \R^N\) be a closed convex set. Then, for every \(z\in\R^N\) and \(y\in C\),
	\[
	y=\mathbb P_C(z)
	\qquad\Longleftrightarrow\qquad
	z-y\in \partial I_C(y)=\N_y C.
	\]
\end{lem}

\bigskip

\begin{lem}{\label{lem:incl-HSN}}
	Let $\bx$ be given by formula~\eqref{eq:proj-micro-x}. Then, for all $t\geq 0$, we have the inclusion 
	\begin{equation}
	\H_{\bx(t)}^N \subset  \mathcal{T}_{\bx(t)} K^N \cap \big[\bx^0 + t\bu^0 - \bx(t)\big]^\perp.
	\end{equation}
\end{lem}

\begin{proof}
	We have immediately the inclusion $\H_{\bx(t)}^N \subset  \mathcal{T}_{\bx(t)} K^N$. 
	Let us now show the second inclusion $\H_{\tx(t)}^N  \subset  \big[\bx^0 + t\bu^0\big]^\perp$.
	Let $\bw \in \H_{\bx(t)}^N $, by the first inclusion there exists $\by(t) \in K^N$ and $\theta>0$ such that $\bw = \theta\big(\by(t) -\bx(t)\big)$ and
	\begin{align*}
	\langle \bx^0 + t\bu^0- \bx(t), \bw\rangle
	& = \theta \langle  \bx^0 + t\bu^0- \bx(t), \by(t) - \bx(t) \rangle  \leq 0,
	\end{align*}
	using the fact that $\bx(t)$ is the projection of  $\bx_0 + t\bu_0$ onto $K^N$. 
	Since $\H_{\bx(t)}^N$ is a linear subspace, if $\H_{\bx(t)}^N$ then $-\bw \in \H_{\bx(t)}^N$. 
	Hence, substituting $\theta$ by $-\theta$ in the previous inequalities, we get
	\[
	\langle \bx^0 + t\bu^0 - \bx(t), \bw\rangle = 0.
	\]
\end{proof}

\bigskip
\begin{lem}\label{prop:proj-velo-micro}
	Let $\bx$ be given by formula~\eqref{eq:proj-micro-x}. Then, for almost all $t \geq 0$, we can define $\bu(t) = \dot{\bx}(t)$.
	The velocity $\bu(t)$ is then the orthogonal projection of $\bu^0$ onto $\H_{\tx(t)}^N$.
	As a consequence, $\bu(t)$ is constant on the "intervals" where the translated vector $\tx(t)$~\eqref{df:translat-x} is constant. 
	Namely, writing
	\begin{equation*}
	\Omega_{\bx(t)}^N = \bigcup_{k} \ \mathbb{J}_k(t), \quad \text{ with } \quad \mathbb{J}_k(t) = \llbracket j_k(t), j_k(t) + L_k(t) \rrbracket,\quad \text{for some }  \ j_k(t),L_k(t) \in \llbracket 1, N\rrbracket,
	\end{equation*} 
	we have
	\begin{equation}\label{eq:proj-u-micro}
	u_i(t) = \begin{cases}
	u^0_i \quad &\text{if}~ i \notin \Omega_{\bx(t)}^N,\\
	\displaystyle \dfrac{1}{L_k} \sum_{l=1}^{L_k} u^0_{j_k+l}  &\text{if}~ i \in \mathbb{J}_k.
	\end{cases}
	\end{equation}	
\end{lem}

\begin{proof}
	First, it is not difficult to see that $\bu(t) \in \H_{\bx(t)}^N$. Indeed, for $j \in \Omega_{\bx(t)}^N$, we have $x_{j+1}(t) = x_i(t) + 2r$ so that, by time differentiation, for a.a. $t$
	\[
	u_{i+1}(t) = \dfrac{d}{dt} x_{i+1}(t) = \dfrac{d}{dt} x_{i}(t) = u_i(t).
	\] 
	We split the rest of the proof into two parts.
	\begin{itemize}
		\item[ Step 1.] $\langle \bu^0 - \bu(t) , \bu(t) \rangle \geq 0$\\
		For $h>0$, we introduce the approximate velocity $\hat u_h(t) := \dfrac{\bx(t+h) - \bx(t)}{h}$. Note that there exists a sequence $(h_n)_n \to 0$ such that $(\hat u_n(t))_n = (\hat u_{h_n}(t))_n \to \bu(t)$.\\
		Since  $\bx(t+ h_n)$ is the projection of  $\bx^0 + (t+ h_n)\bu^0$ onto $K^N$, we have
		\begin{align*}
		0 & \geq \langle \bx^0 + (t+ h_n)\bu^0 - \bx(t+h_n) , \bx(t) - \bx(t+h_n) \rangle \\
		& \geq  - \langle \bx^0 + t\bu^0 - \bx(t) , \bx(t+h_n) - \bx(t) \rangle \\
		& \qquad - h_n \langle \hat \bu_n(t) -\bu^0, \bx(t) - \bx(t+h_n) \rangle.
		\end{align*}
		Since the first term is non-negative ($\bx(t)$ is the projection of  $\bx^0 + t\bu^0$ onto $K^N$), we deduce that
		\[
		- \langle \bu^0 - \hat \bu_n, \hat\bu_n\rangle = - h_n \langle \hat \bu_n(t) -\bu^0, \bx(t) - \bx(t+h_n) \rangle \leq 0.
		\]
		We conclude by passing to the limit $h_n\to 0$.
		\item[ Step 2.]  $\langle \bu^0 - \bu(t) ,\bw  \rangle \leq 0$ for all $\bw \in \H_{\bx(t)}^N$. \\
		Let $\bw = \theta(\by(t) -\bx(t)) \in \H_{\bx(t)}^N$ with $\theta\geq0$ and $\by \in K^N$.
		We also know by the previous lemma that
		\begin{equation}\label{eq:w-Hs-perp}
		\langle \bx^0 + t\bu^0, \bw\rangle = 0
		\end{equation}
		As previously we start with the fact  $\bx(t+ h_n)$ is the projection of  $\bx^0 + (t+ h_n)\bu^0$ onto $K^N$ to write
		\begin{align*}
		0 & \geq \langle \bx^0 + (t+ h_n)\bu^0 - \bx(t+h_n), \by(t) - \bx(t+h_n) \rangle.
		\end{align*}
		Since $\bx(t+h_n) = \bx(t) + h \hat \bu_n(t)$ we have on the hand
		\begin{align*}
		\bx^0 + (t+ h_n)\bu^0 - \bx(t+h_n)
		& =  \bx^0 + t \bu^0 - \bx(t) + h_n(\bu^0 - \bu(t)) - h_n(\hat \bu_n(t) - \bu(t)).
		\end{align*}
		and on the other hand
		\[
		\by(t) - \bx(t+h_n) 
		= \by(t) - \bx(t) - h_n \bu(t) - h_n(\hat \bu_n(t) - \bu(t)).
		\]
		Therefore, reordering the terms in terms of powers of $h_n$, we have 
		\begin{align*}
		0 
		& \geq   \langle \bx^0 + t\bu^0 - \bx(t) , \ty(t) - \tx(t) \rangle
		- h_n \langle \bx^0 + t\bu^0 - \bx(t) , \bu(t) \rangle \\
		& \quad    + h_n \langle \bu^0(t) - \bu(t) , \ty(t) - \tx(t) \rangle \\
		& \quad	 - h_n \Big[\langle \bx^0 + t\bu^0 - \bx(t) , \hat \bu_n(t) - \bu(t) \rangle 
		+ \langle \hat \bu_n(t) - \bu(t) , \ty(t) - \tx(t) \rangle \Big]\\
		& \quad - h_n^2 \Big[\langle \bu^0 - \bu(t),\bu(t)\rangle + \langle \bu^0 - \bu(t), \hat\bu_n(t) -\bu(t)\rangle \\
		& \qquad \qquad			+ \langle   \hat\bu_n(t) -\bu(t), \bu(t) \rangle - N\| \hat\bu_n(t) -\bu(t)\|^2
		\Big]	 
		\end{align*}
		where the two terms of the first line vanish. 
		As a consequence, by dividing by $h_n$ we get
		\begin{align*}
		\langle \bu^0(t) - \bu(t) , \by(t) - \bx(t) \rangle
		& \leq \Big[\langle \bx^0 + t\bu^0 - \bx(t) , \hat \bu_n(t) - \bu(t) \rangle 
		+ \langle \hat \bu_n(t) - \bu(t) , \by(t) - \bx(t) \rangle \Big]\\
		& \quad + h_n \Big[\langle \bu^0 - \bu(t),\bu(t)\rangle + \langle \bu^0 - \bu(t), \hat\bu_n(t) -\bu(t)\rangle \\
		& \qquad \qquad			+ \langle   \hat\bu_n(t) -\bu(t), \bu(t) \rangle - N\| \hat\bu_n(t) -\bu(t)\|^2 \Big].
		\end{align*}
		Letting eventually $h_n\to 0$, we obtain the desired inequality:
		\begin{equation}\label{eq:ineg-projortho}
		\langle \bu^0(t) - \bu(t) , \by(t) - \bx(t) \rangle \leq 0 \qquad \forall \ \by \in K^N.
		\end{equation}
	\end{itemize}
\end{proof}

\begin{proof}[Conclusion of the proof of Proposition~\ref{prop:proj-micro}]
In view of the previous lemmas, the formula~\eqref{eq:proj-micro-x} defines for a.a. $t$ a velocity $u(t) \in \H_{\bx(t)}^N$.

\medskip
{\it Definition and properties of the multipliers.}	
For almost every $t$, we introduce the vector
\begin{equation}
\xi(t) := \bu(t) - \bu^0 \in \R^N.
\end{equation}
From~\eqref{eq:ineg-projortho}, we obtain
\[
-\langle \xi(t) , \ty(t) - \tx(t) \rangle \leq 0, \quad \forall \ \ty \in \tK^N
\]
which implies 
\[
-\xi(t) \in \partial I_{K^N}(\bx(t)) = \N_{\bx(t)}K^N.
\]
Therefore the relation~\eqref{eq:micro-f0-lagr} is satisfied.\\
We now define the coefficients $\lambda_i$ recursively. 
Set $\lambda_0(t) = 0$ and for $i=1, \dots N$,
\begin{equation*}
\lambda_i(t) = \lambda_{i-1}(t) - \dfrac{1}{N}\xi_i(t) \quad i=1, \dots,N.
\end{equation*}
With this definition, we obtain
\begin{equation}
u_i(t) - u^0_i = -N\big(\lambda_i(t) - \lambda_{i-1}(t)\big)  \quad i=1, \dots,N.
\end{equation}
The quantity $\lambda_i(t)$ can be interpreted as the reaction force associated with the contact between particle $i$ and particle $i+1$.
In principle, the coefficients $\lambda_i$ could be introduced directly as the Karush-Kuhn-Tucker multipliers associated to the non-overlap constraints (see Appendix~\ref{app:KKT}).
In that framework one would obtain the condition
\[
\lambda_i(t) \geq 0, \qquad (x_{i+1}(t) - x_i(t) - 2r) \lambda_i(t) = 0.
\] 
However, these properties can be easily derived as follows (we will re-use these calculations to obtain properties on the time derivative of the $\lambda_i$ in Proposition~\ref{prop:micro-pressure}):
using the expression~\eqref{eq:proj-u-micro} of the velocity $\bu(t)$ and the definition of the translated vector $\tx$~\eqref{df:translat-x}, we compute
\begin{align*}
\langle \bu(t) - \bu^0 , \tx(t) \rangle
& = \sum_{i=1}^N \big(u_i(t) - u^0_i\big) \tilde x_i(t) \\
& = \sum_{k=1,\dots,\alpha(t)} \sum_{i \in \mathbb{J}_k(t)} \Big(\frac{1}{L_k} \sum_{l=1}^{L_k} u^0_{j_k +l} - u^0_i\Big) \tilde x_i(t) \\
& = \sum_{k=1,\dots,\alpha(t)} \Big(\sum_{l=1}^{L_k} u^0_{j_k +l} - \sum_{i \in \mathbb{J}_k(t)} u^0_i\Big) \tilde x_{j_k}(t)\\
& = 0,
\end{align*}
since $\tilde x_i(t) = \tilde x_{j_k}(t)$ for all $i \in \mathbb{J}_k(t) = \llbracket j_k(t), j_k(t) + L_k(t) \rrbracket$.
Consequently,
\[
\langle -\sum_{i=1}^{N-1} \lambda_i(t)(e_i - e_{i+1}) , \tx(t) \rangle  = 0,
\]	
which shows that $\lambda_i(t)$ may be nonzero only when the corresponding constraint is active.
Next, let $\ty \in \tK^N$ satisfy 
\[
\langle \bx^0 + t \bu^0 - \bx(t), \ty - \tx(t) \rangle = 0.
\] 
Using~\eqref{eq:ineg-projortho} (see the proof of Prop.~\ref{prop:proj-velo-micro}), we obtain
\begin{align*}
\langle N\sum_{i=1}^{N-1} \lambda_i(t)(e_i - e_{i+1}) , \by \rangle  
& =\langle \bu^0 -\bu(t) , \ty - \tx(t) \rangle  \leq 0,
\end{align*}
that is
\[
N\sum_{i=1}^{N-1} \lambda_i(t) \big(\tilde y_i - \tilde y_{i+1}\big) \leq 0.
\]
Because $\tilde y_i - \tilde y_{i+1} \leq 0$ for every $\ty \in \tK^N$, and $\ty$ is here arbitrary, we conclude that
\[
\lambda_i(t) \geq 0.
\]

\medskip	
{\it Uniqueness.}
Let \(\bx^1,\bx^2\) be two solutions of the first-order differential inclusion~\eqref{eq:micro-f0-lagr}. Then there exist selections
\[
\xi^a(t)\in \partial I_{K^N}(\bx^a(t)),
\qquad a=1,2,
\]
such that
\[
\dot\bx^a(t)+\xi^a(t)=\bu^0
\qquad\text{for a.e. } t \geq 0.
\]
Subtracting the two identities gives
\[
\frac{d}{dt}(\bx^1-\bx^2)+(\xi^1-\xi^2)=0.
\]
Taking the scalar product with $\bx^1-\bx^2$, we find
\[
\frac{N^2}{2}\frac{d}{dt}\|\bx^1-\bx^2\|^2
=
-\langle \xi^1-\xi^2,\bx^1-\bx^2\rangle.
\]
Since $K^N$ is convex, the subdifferential $\partial I_{K^N}$ is a monotone operator. 
Indeed, since $\xi^a \in \partial I_{K^N}(\bx^a)$, $a=1,2$, we have
\[
\langle \xi^1,\bx^2 - \bx^1 \rangle \leq 0,
\qquad
\langle \xi^2,\bx^1 - \bx^2\rangle \leq 0.
\]
Adding the two inequalities yields the monotonicity of $\partial I_{K^N}$., i.e.
\[
\langle \xi^1 - \xi^2 , \bx^1 - \bx^2\rangle \geq 0,
\]
so that
\[
\frac{d}{dt}\|\bx^1-\bx^2\|^2\le0.
\]
Using $\bx^1(0)=\bx^2(0)=\bx^0$, we conclude that $\bx^1(t)=\bx^2(t)$ for all $t\ge0$. 
Therefore the solution is unique.

\end{proof}

\subsection{Estimates for the microscopic variables}

In this subsection we establish several uniform estimates for the microscopic variables. 
These bounds will play a crucial role in the compactness analysis performed in the next section.

\begin{prop}\label{prop:bound-micro}
	Consider the solution $\bx$ given by the projection formula~\eqref{eq:proj-micro-x}.
	Let 
	\[
	L(t) : = (\lambda_1(t), \dots, \lambda_N(t)) \qquad 
	\Delta L(t) := (\lambda_1(t)-\lambda_0, \dots, \lambda_N(t)-\lambda_{N-1}(t)).
	\]
	Then, the following estimates hold
	\begin{align}
	\|\bx\|_{L^\infty(0,T;\R^N)}  + \|\bu\|_{L^\infty(0,T;\R^N)} 
	+ \|L\|_{L^\infty(0,T;\R^N)} + \|N\Delta L\|_{L^\infty(0,T;\R^N)}  \leq C(T,\|\bx^0\|, \|\bu^0\|),
	\end{align}	
	for some constant $C>0$ independent of $N$.
\end{prop}

\begin{proof}
We first use the contraction property of the projection operator onto the convex set $K^N$.
From the projection formula~\eqref{eq:proj-micro-x} we obtain
\[
\|\bx(t)\| \leq \|\bx^0 + t\bu^0\| \leq \|\bx^0\| + T \|\bu^0\|.
\]	
Similarly (see for instance~\eqref{eq:proj-u-micro}), the velocity satisfies
\[
\|\bu(t)\| \leq \|\bu^0\|.
\]
Next, using the relation 
\[
N(\lambda_i(t) - \lambda_{i-1}(t)) = u_i(t) - u^0_i,
\]
we obtain
\begin{align*}
\|\Delta L(t)\|^2 
& = \sum_{i=1}^N \dfrac{1}{N}(\lambda_i(t)-\lambda_{i-1}(t))^2 = \sum_{i=1}^{N} \dfrac{1}{N^3}(u_i(t)-u^0_i)^2 \leq \dfrac{4}{N^2}\|\bu^0\|^2.
\end{align*}
Finally by the Cauchy-Schwarz inequality,
\begin{align*}
|\lambda_i(t)| & \leq \sum_{j\leq i} |\lambda_j(t) - \lambda_{j-1}| 
\leq N \left(\sum_{j=1}^N \dfrac{|\lambda_j(t) - \lambda_{j-1}|^2}{N}\right)^{1/2} 
= N  \|\Delta L(t)\| .
\end{align*}
Hence
\begin{align*}
\|L(t)\|^2 
& = \sum_{i=1}^N \dfrac{|\lambda_i(t)|^2}{N}
\leq N^2 \|\Delta L(t)\|^2 \leq 4 \|\bu^0\|^2.
\end{align*}
This concludes the proof.
\end{proof}

\medskip
While the above proposition provides uniform bounds in the rescaled Euclidean norm, we now derive stronger pointwise estimates on the velocities and the multipliers.
The following lemma provides uniform pointwise bounds on the microscopic variables based on the additional assumption that the initial velocities are uniformly bounded. 

\medskip
\begin{lem}
	Let us assume that 
	\[\sup_i |u^0_i| \leq C_0'.\]
	Then, we have
	\begin{align}\label{eq:sup-micro-u-L}
	\sup_{t \in [0,T]} \Big[\sup_i |u_i(t)| + \sup_i \left|N\big(\lambda_i(t) - \lambda_{i-1}(t)\big) \right| + \sup_i N |\lambda_i(t)|\Big] \leq C', 
	\end{align}
	for some positive $C'$ independent of $N$.
\end{lem}

\medskip
The estimates follow directly from the projection formula and the relation~\eqref{eq:dyn-micro}.

\medskip

The next result provides a discrete Oleinik-type estimate, which will be essential in the compactness analysis leading to the macroscopic limit.

\medskip

\begin{prop}[Oleinik-type estimate] {\label{prop:Oleinik}}
The following estimate holds
\begin{equation}\label{eq:Oleinik-micro}
\dfrac{u_i(t) - u_{i-1}(t) }{x_{i}(t) - x_{i-1}(t)} < \dfrac{1}{t} \qquad \forall \ t > 0, \ i \in \llbracket 1,N \rrbracket.
\end{equation}
\end{prop}

\medskip
\begin{proof}
Initially, we have
\[
x_{i}(0) \geq x_{i-1}(0) + 2r > x_{i-1}(0).
\]
Fix $t>0$ such that particles $i$ and $i-1$ are not in contact at $t$.
Otherwise, $u_i(t) = u_{i-1}(t)$ and the inequality is trivially satisfied.
By the stickiness of the dynamics, the particles cannot cross and clusters can only grow in time. 
In particular, since particles $i$ and $i-1$ are not in contact at time $t$, they have not collided on the time interval $(0,t)$.
Define 
\[
x_{\ell} := x_{i-1}(t) - t u_{i-1}(t), \qquad x_{r} = x_{i}(t) - t u_{i}(t).
\]
By the stickiness property, the velocity $u_{i-1}$ is nondecreasing on $(0,t)$ (collision with the particle $i-2$), while $u_{i}$ is nonincreasing $(0,t)$ (collision with particle $i+1$).
Hence
\[
x_{\ell} \leq x_{i-1}(0), \qquad x_{r} \geq x_{i}(0).
\] 
Combining that with the initial separation of $i$ and $i-1$ we get
\[
x_{i-1}(t) - t u_{i-1}(t) <  x_{i}(t) - t u_{i}(t).
\]
This concludes the proof.
\end{proof}


\subsection{Second-order formulation and microscopic pressure}

We now return to the original second-order formulation of the dynamics. 
The first-order inclusion obtained in Proposition~\ref{prop:proj-velo-micro} admits a semigroup structure, which allows us to recover a second-order differential inclusion and to introduce the associated microscopic pressure.

\begin{lem}[Semi-group property]
	For any $0 \leq s < t$, the solution satisfies
	\begin{align}\label{eq:semigroup}
	\bx(t) = \mathbb{P}_{K^N}\big(\bx(s) + (t-s)\bu(s)\big),\\
	\bu(t) = \mathbb{P}_{\H_{\bx(t)}^N}\big(\bu(s)\big).\label{eq:projortho-sg}
	\end{align}
\end{lem}

\medskip
\begin{proof}
The result follows by repeating the arguments of Lemmas~\ref{lem:incl-HSN} and~\ref{prop:proj-velo-micro}, replacing the initial data $(\bx^0,\bu^0)$ by $(\bx(s),\bu(s))$.
\end{proof}

\medskip
We now use the semigroup property to recover the second-order formulation of the dynamics. 
Formally, this amounts to differentiating the relation \eqref{eq:dyn-micro} with respect to time.
This leads to the introduction of a microscopic pressure, defined as the time derivative of the multipliers $\lambda_i$.

\medskip

\begin{prop}[Definition of the microscopic pressure]\label{prop:micro-pressure}
	For each $i=1,\dots,N$, define
	\begin{equation}
	p_i = \dfrac{d}{dt}\lambda_i \in W^{-1,\infty}(0,T).
	\end{equation}	
	Then, Equation~\eqref{eq:micro-2nd-ID} holds in the sense of distributions.
	Moreover, as a consequence of the semigroup (or stickiness) property, the pressure satisfies the same complementarity conditions as the multipliers:
	\begin{equation}
	(x_{i+1} - x_i - 2r) p_i = 0, \qquad p_i \geq 0 \quad \text{in }\ \mathcal{D}'(0,T).
	\end{equation}
\end{prop}

\medskip

\begin{proof}
	We follow the same computations as in the proof of Proposition~\ref{prop:proj-micro} using the semigroup property~\eqref{eq:semigroup}. 
	For any $h> 0$,	we have
	\begin{align*}
	N \langle \sum_{i=1}^{N-1} \big(\lambda_i(t) - \lambda_i(t-h) \big) (e_i - e_{i+1})  , \tx(t) \rangle 
	= - \langle \bu(t) - \bu(t-h) , \tx(t) \rangle 
	& = 0,
	\end{align*}
	Dividing by $h$, and letting $h \to 0$ we obtain
	\[
	(x_{i+1} - x_i - 2r) p_i = 0 \quad \text{in the sense of distributions},
	\] 
	using that $\lambda_i \in L^\infty(0,T)$, and $\dfrac{d}{dt} x_i = u_i\in L^\infty(0,T)$).\\
	Next, let $\by \in K^N$ satisfy 
	\[
	\langle \bx(t-h) + h \bu(t-h) - \bx(t), \by - \bx(t) \rangle = 0.
	\]
	Using~\eqref{eq:projortho-sg}, we obtain on the associated translated vector $\ty \in \tK^N$
	\begin{align*}
	N \langle \sum_{i=1}^{N-1} \big(\lambda_i(t) - \lambda_i(t-h) \big) (e_i - e_{i+1})  , \ty  \rangle 
	= - \langle \bu(t) - \bu(t-h) , \ty - \tx(t) \rangle & \leq 0
	\end{align*}
	Since $\ty \in \tK^N$ is otherwise arbitrary, this implies
	\[
	N\sum_i \big(\lambda_i(t) - \lambda_i(t-h) \big) \big(\underset{ \leq 0}{\underbrace{\tilde y_i - \tilde y_{i+1}}}\big) \leq 0,
	\]
	and therefore
	\[
	\lambda_i(t) - \lambda_i(t-h) \geq 0.
	\]
	Dividing by $h$ and passing to the limit $h \to 0$, we conclude that
	\[
	p_i\geq 0 \quad \text{in the sense of distributions}.
	\]
\end{proof}

\section{Hydrodynamic limit}\label{sec:limitN}

We now turn to the mean-field limit of the microscopic system as $N \to +\infty$. 
To this end, we introduce suitable interpolations of the discrete variables, which allow us to describe the particle system as functions defined on a fixed spatial domain.
These interpolated quantities will satisfy uniform estimates, and we will show that they converge (up to subsequences) towards a solution of the limiting macroscopic system~\eqref{eq:macro-euler}.

The section is organized as follows. 
In Section~\ref{ssec:interpol1}, we introduce the interpolated variables and establish their main properties. 
In Section~\ref{ssec:limit1}, we derive uniform bounds and compactness 
results. Finally, in Section~\ref{ssec:cvg2}, we pass to the limit and identify the limiting system.

\subsection{Interpolation of the microscopic variables}\label{ssec:interpol1}

In order to pass to the limit as $N \to +\infty$, we introduce suitable representations of the discrete variables $(x_i,u_i,\lambda_i)$.
More precisely, we define piecewise constant (or piecewise affine) functions on the interval $(0,1)$, which encode the position, velocity, and Lagrange
multipliers of the particle system. 

\medskip
We introduce the Lagrangian mass variable $w \in (0,1)$ which corresponds to the cumulative mass coordinate of the system. 
In the discrete setting, the particles are naturally associated with the intervals
\[
\left(\frac{i-1}{N}, \frac{i}{N}\right], \quad i=1,\dots,N,
\]
so that $w$ can be interpreted as a continuous counterpart of the particle index.

\bigskip
We define the position function $X_N : (0,T)\times(0,1)\to \mathbb{R}$ by
\begin{equation}\label{df:XN}
X_N(t,w) := x_i(t)
\quad \text{for } w \in \left(\frac{i-1}{N}, \frac{i}{N}\right].
\end{equation}
Similarly, we define the velocity function
\begin{equation}\label{df:UN}
U_N(t,w) := u_i(t)
\quad \text{for } w \in \left(\frac{i-1}{N}, \frac{i}{N}\right].
\end{equation}
By construction, we have
\begin{equation}
\partial_t X_N = U_N
\quad \text{in } L^\infty((0,T)\times(0,1)).
\end{equation}
In order to describe the constraint forces, we introduce the interpolated multiplier
\begin{equation}\label{df:LN}
\Lambda_N(t,w) := \lambda_i(t)
\quad \text{for } w \in \left(\frac{i-1}{N}, \frac{i}{N}\right].
\end{equation}
We then define the microscopic pressure as
\begin{equation}\label{df:PN}
P_N := \partial_t \Lambda_N,
\end{equation}
which is understood in the sense of distributions in time.

\medskip

In order to obtain a well-defined spatial derivative, we introduce also the continuous piecewise affine interpolation of the positions, velocities and multipliers.
We set
\begin{equation}
\tX_N(t,w) = N \big(x_i(t) - x_{i-1}(t)\big)\left(w - \dfrac{i-1}{N}\right) + x_{i-1}(t) , ~  w \in \left[\dfrac{i-1}{N},\dfrac{i}{N}\right], \ i=1, \dots,N,
\end{equation}
by adding an artificial particle on the left 
\begin{equation}\label{df:x0}
x_0(t) = x_1(t) - \dfrac{1}{N} - \delta, \quad \text{with} \quad \delta > 0 \quad \text{independent of} \quad N.
\end{equation}
Doing so we preserve the non-overlapping constraint and the boundary condition $\lambda_0 = 0$.  
We will also check later that considering $\delta > 0$ does not alter the convergence of $\tX_N$ as $N \to +\infty$, and we will use this crucial parameter $\delta$ to derive a uniform bound on the pressure (see the proof of Proposition~\ref{prop:pressure-bound}).\\
Observe that 
\begin{equation*}
\partial_w\tX_N(t,w) = N(x_i(t) - x_{i-1}(t)) \geq 2rN = 1 \quad \text{ for } w \in \left[\dfrac{i-1}{N},\dfrac{i}{N}\right].
\end{equation*}
Similarly, we define the continuous piecewise affine interpolations of the velocity and the multipliers: for  $w \in \left[\dfrac{i-1}{N},\dfrac{i}{N}\right]$, we set
\begin{align}
\tU_N(t,w)  & = N \big(u_i(t) - u_{i-1}(t)\big)\left(w - \dfrac{i-1}{N}\right) + u_{i-1}(t),\\
\tLambda_N(t,w) & = N \big(\lambda_i(t) - \lambda_{i-1}(t)\big)\left(w - \dfrac{i-1}{N}\right) + \lambda_{i-1}(t),\\
\tP_N(t,w) & = N \big(p_i(t) - p_{i-1}(t)\big)\left(w - \dfrac{i-1}{N}\right) + p_{i-1}(t).
\end{align}

\medskip

\begin{lem}[Discrete PDE system - Order 1]
	The interpolated quantities satisfy the following system in the sense of distributions on $(0,T)\times(0,1)$:
	\begin{subnumcases}{\label{eq:PDE-Lagr-N-1}}
	\partial_t X_N = U_N^0 - \partial_w \tLambda_N, \\
	\partial_w \tX_N \geq 1, \quad (\partial_w \tX_N -1)\,\Lambda_N = 0,\quad \Lambda_N \geq 0.
	\end{subnumcases}
	These equations are supplemented with the initial condition
	\[
	X_N(0,w) = X_N^0(w),
	\]
	where $X_N^0$ is the corresponding interpolation of the initial data.
\end{lem}

\medskip

\begin{proof}
	From the identity
	\[
	u_i(t) 
	= u_i^0 - N\big(\lambda_i(t) - \lambda_{i-1}(t)\big),
	\]
	we obtain
	\[
	\partial_t X_N = U_N^0 - \partial_w \tLambda_N
	\]
	in the sense of distributions.
	The complementarity relation
	\[
	(x_{i+1} - x_i - 2r)\, \lambda_i = 0
	\]
	translates into
	\[
	(\partial_w \tX_N - 1)\,\Lambda_N = 0.
	\]
	Finally, the nonnegativity $\Lambda_N \ge 0$ follows from that of $\lambda_i$.
\end{proof}

\bigskip

\begin{lem}[Discrete PDE system- Order 2]\label{lem:discPDE2}
	The interpolated quantities satisfy the following system in the sense of distributions on $(0,T)\times(0,1)$:
	\begin{subnumcases}{\label{eq:PDE-Lagr-N-2}}
	\partial_t X_N = U_N, \\
	\partial_t U_N + \partial_w \tP_N = 0, \\
	\partial_w \tX_N \geq 1, \quad (\partial_w \tX_N -1)\,P_N = 0,\quad P_N \geq 0.
	\end{subnumcases}
	These equations are supplemented with the initial conditions
	\[
	X_N(0,w) = X_N^0(w),
	\qquad
	U_N(0,w) = U_N^0(w),
	\]
	where $\tX_N^0$ and $U_N^0$ are the corresponding interpolations of the initial data.
\end{lem}

\subsection{Uniform estimates}\label{ssec:limit1}

Our goal is now to pass to the limit $N \to +\infty$ in the discrete system.
To this end, we first establish uniform bounds on the interpolated quantities $(X_N,\tX_N,\Lambda_N,\tLambda_N)$.
These estimates are inherited from the microscopic dynamics and will provide the compactness needed to identify the limiting system.

\begin{prop}[Uniform estimates - $L^2$ bounds]\label{prop:L2-micro}
	Assume that the initial data satisfy
	\[
	\|\bx^0\| + \|\bu^0\| \leq C_0,
	\]
	for some $C_0$ independent of $N$. 
	Then, the following bounds hold:
	\begin{align}
	\|X_N\|_{L^\infty(0,T;L^2(0,1))} + \|\tX_N\|_{L^\infty(0,T;L^2(0,1))} + \|\partial_t X_N\|_{L^\infty(0,T;L^2(0,1))} + \|\partial_t \tX_N\|_{L^\infty(0,T;L^2(0,1))} \notag \\
	+ \|\Lambda_N\|_{L^\infty(0,T;L^2(0,1))} +\|\tLambda_N\|_{L^\infty(0,T;L^2(0,1))} + \|\partial_w \tLambda_N\|_{L^\infty(0,T;L^2(0,1))} \leq C,
	\end{align}
	for some $C = C(C_0,T) >0$ independent of $N$.
\end{prop}

These estimates follow directly from Proposition~\ref{prop:bound-micro} and the definitions of the interpolations.

\medskip

We now derive stronger estimates under an additional assumption on the initial data, namely a uniform bound on the initial velocities together with a uniform control 
of the spatial spread of the initial configuration. The latter follows from the compact support of the initial density $\rho^0$, which ensures that the discrete initial configurations remain uniformly confined.
Under these assumptions, we obtain uniform $L^\infty$ and BV bounds on the interpolated variables.

\begin{prop}[Uniform estimates - $L^\infty$-bounds]\label{prop:Linf-micro}
	Assume the conditions of the previous proposition. 
	In addition, assume that the initial data satisfy
	\begin{equation}\label{cond:Linfty}
	x^0_N - x_0^0 + \sup_i |u^0_i| \leq C_0'
	\end{equation}
	with $C_0'>0$ independent of $N$.
	Then, the following BV estimate holds:
	\begin{align}\label{eq:BV_XN}
	& \sup_{t \in [0,T]} \int_0^1 |\partial_w \tX_N(t,w)| dw 
	\leq C',
	\end{align}
	for some positive constant $C'= C'(T, C_0')$.\\
	Without loss of generality, we may assume $x^0_0 = 0$ for all $N$.
	As a consequence, we have the uniform bounds
	\begin{equation}
	\|X_N\|_{L^\infty((0,T)\times(0,1))} + \|\tX_N\|_{L^\infty((0,T)\times(0,1))} \leq C',
	\end{equation}
	and
	\begin{equation}
	\|\Lambda_N\|_{L^\infty((0,T)\times(0,1))} + \|\tLambda_N\|_{L^\infty((0,T)\times(0,1))} \leq C',
	\end{equation}
	for some positive constant $C'= C'(T, C_0')$.
\end{prop}

\medskip
\begin{proof}
The BV bound derives from the following inequality:
\begin{align*}
& \int_0^1 |\partial_w \tX_N(t,w)| dw = x_N(t) - x_0(t) \\
&\leq (x_N(0) - x_0(0)) + T\left(\sup_{[0,T]}|u_N(t)| + \sup_{[0,T]}|u_1(t)|\right) 
\leq C',
\end{align*}
where we have used the $L^\infty_t$ bound on the velocities~\eqref{eq:sup-micro-u-L} and the initial assumption~\eqref{cond:Linfty}.\\
The $L^\infty$ bound on $X_N$ and $\tX_N$ follows directly from the control
of the spatial spread, up to fixing $x_0(0)=0$.

Finally, using the relation
\[
N(\lambda_i(t) - \lambda_{i-1}(t)) = u_i(t) - u_i^0,
\]
together with the uniform bound on the velocities, we obtain a uniform bound on the multipliers, which implies
\[
\|\Lambda_N\|_{L^\infty} + \|\tLambda_N\|_{L^\infty} \le C.
\]
This concludes the proof.	
\end{proof}

\medskip

Thanks to the uniform estimates established above, we are now in a position to derive compactness properties for the sequence of interpolated variables.
These compactness results will allow us to extract converging subsequences and identify the limit of the first-order differential inclusion~\eqref{eq:PDE-Lagr-N-1}.

\begin{prop}[Compactness and limit in the first-order dynamics]\label{prop:cvg1}
Assume that the hypotheses of Proposition~\ref{prop:Linf-micro} hold. Assume moreover that there exist $(X^0, U^0) \in \big(L^2(0,1)\big)^2$, with $X^0$ monotone (non-decreasing) with $\partial_w X^0 \geq 1$ a.e., such that
\begin{align}
\| X^0_N - X^0\|_{L^2(0,1)} \to 0 \quad \text{as}~ N \to +\infty, \\
\| U^0_N - U^0\|_{L^2(0,1)} \to 0 \quad \text{as}~ N \to +\infty.
\end{align}
Then, there exists 
\[
X \in W^{1,\infty}((0,T)\times(0,1)), \qquad \Lambda \in L^\infty(0,T; W^{1,\infty}(0,1)),
\]
such that, up to the extraction of a subsequence,
\begin{align*}
\tX_N & \to X \quad \text{strongly in} ~ \mathcal{C}([0,T]\times [0,1]),\\
X_N  & \to X \quad \text{strongly in} ~  \mathcal{C}([0,T];L^1(0,1)) \quad \text{and weakly-* in} ~ L^\infty((0,T)\times(0,1)) ,\\
\Lambda_N, \ \tLambda_N  &  \rightharpoonup^* \Lambda \quad \text{weakly-* in} ~ L^\infty((0,T)\times(0,1)),\\
\partial_w \tLambda_N  &  \rightharpoonup^* \partial_w\Lambda \quad \text{weakly-* in} ~ L^\infty((0,T)\times(0,1)).
\end{align*}
Moreover, the limit $(X,\Lambda)$ satisfies
\begin{equation}\label{eq:Lagr-1}
\begin{cases}
\partial_t X = U^0 -  \partial_w \Lambda ,\\
\partial_w X \geq 1, \quad \big(\partial_w X - 1\big)\Lambda = 0, \quad \Lambda \geq 0, \quad \text{a.e.},\\
\Lambda(t,0) = \Lambda(t,1) = 0,\quad \text{a.e.}.
\end{cases}
\end{equation}
with the initial condition $X(0, \cdot)=X^0$.	
\end{prop}

\medskip
\begin{proof}
We divide the proof into several steps.

\begin{itemize}
	\item From the uniform $L^\infty(0,T;L^2(0,1))$ bounds established in Proposition~\ref{prop:Linf-micro}, we deduce that, up to extraction of a subsequence,
	\[
	X_N \rightharpoonup^* X,
	\qquad
	\tX_N \rightharpoonup^* \tX
	\quad \text{in } L^\infty((0,T)\times(0,1)).
	\]
	\item The $L^\infty$ bound together with the BV estimate of	Proposition~\ref{prop:Linf-micro} ensures that $\partial_w \tX_N$ is bounded in $L^\infty(0,T;L^1(0,1))$,
	while $\partial_t \tX_N = \tU_N$ is bounded in $L^\infty$. 
	Hence, by the Arzelà–Ascoli theorem, we obtain
	\[
	\tX_N \to X \quad \text{strongly in } C([0,T]\times[0,1]).
	\]
	
	Moreover, we have the estimate
	\[
	\int_0^1 |\tX_N(t,w) - X_N(t,w)|\,dw
	= \frac{1}{2N} \sum_{i=0}^{N-1} (x_{i+1}(t) - x_i(t))
	= \frac{1}{2N} (x_N(t) - x_0(t))
	\leq \frac{C}{N},
	\]
	which implies that $X_N - \tX_N \to 0$ in $L^1$, and therefore
	\[
	X_N \to X \quad \text{strongly in } \mathcal{C}([0,T];L^1(0,1)).
	\]
	\item From the uniform $L^\infty$ bounds on $\Lambda_N$ and $\tLambda_N$,
	we extract a subsequence such that
	\[
	\Lambda_N, \tLambda_N \rightharpoonup^* \Lambda
	\quad \text{in } L^\infty((0,T)\times(0,1)).
	\]
	
	Moreover, for $w \in \big(\frac{i-1}{N},\frac{i}{N}\big]$, we have
	\begin{equation}\label{eq:diffLambda_N}
	|\tLambda_N(t,w) - \Lambda_N(t,w)|
	= \left|N(\lambda_i(t) - \lambda_{i-1}(t))
	\left(w - \frac{i-1}{N}\right)\right|
	\le |\lambda_i(t) - \lambda_{i-1}(t)|
	\le \frac{C'}{N},
	\end{equation}
	which shows that $\tLambda_N - \Lambda_N \to 0$ uniformly, and therefore
	both sequences have the same limit.
	
	Similarly, $\partial_w \tLambda_N$ converges weakly-* to $\partial_w \Lambda$.
	\item Passage to the limit in the exclusion relation.
	We start from
	\[
	(\partial_w \tX_N -1)\Lambda_N = 0
	\qquad \text{in } \mathcal D'((0,T)\times(0,1)).
	\]
	Let $\varphi \in \mathcal C_c^1((0,T)\times(0,1))$. Then
	\begin{align*}
	0
	&= \int_0^T \int_0^1 (\partial_w \tX_N -1)\Lambda_N \,\varphi \,dw\,dt \\
	&= \int_0^T \int_0^1 (\partial_w \tX_N -1)(\Lambda_N-\tLambda_N)\,\varphi \,dw\,dt 
	+ \int_0^T \int_0^1 (\partial_w \tX_N -1)\tLambda_N\,\varphi \,dw\,dt .
	\end{align*}
	Using
	\[
	(\partial_w \tX_N)\varphi
	= \partial_w(\tX_N\varphi)-\tX_N\,\partial_w\varphi,
	\]
	and integrating by parts in $w$, we obtain
	\begin{align*}
	0
	&= \int_0^T \int_0^1 (\partial_w \tX_N -1)(\Lambda_N-\tLambda_N)\,\varphi \,dw\,dt \\
	&\quad - \int_0^T \int_0^1 \tLambda_N \,\varphi \,dw\,dt
	-\int_0^T \int_0^1 \tX_N\,\partial_w\tLambda_N\,\varphi \,dw\,dt
	-\int_0^T \int_0^1 \tX_N\,\tLambda_N\,\partial_w\varphi \,dw\,dt \\
	&=: R_N + T_N^1 + T_N^2 + T_N^3 .
	\end{align*}
	
	First, the remainder term tends to zero. Indeed,
	\[
	|R_N|
	\le
	\|\varphi\|_{L^\infty_{t,w}}
	\|\Lambda_N-\tLambda_N\|_{L^\infty_{t,w}}
	\int_0^T \int_0^1 |\partial_w\tX_N-1|\,dw\,dt.
	\]
	Since $\|\Lambda_N-\tLambda_N\|_{L^\infty_{t,w}}\to 0$ and
	$\partial_w\tX_N$ is uniformly bounded in $L^\infty(0,T;L^1(0,1))$,
	we deduce that
	\[
	R_N \to 0
	\qquad\text{as } N\to +\infty.
	\]
	Next, by the weak-* convergence $\tLambda_N \rightharpoonup^* \Lambda$ in
	$L^\infty((0,T)\times(0,1))$, we have
	\[
	T_N^1 \to - \int_0^T \int_0^1 \Lambda\,\varphi \,dw\,dt.
	\]
	Similarly, combining the strong convergence $\tX_N \to X$ in
	$\mathcal C([0,T]\times[0,1])$ with the weak-* convergence
	$\partial_w\tLambda_N \rightharpoonup^* \partial_w\Lambda$, we obtain
	\[
	T_N^2 \to - \int_0^T \int_0^1 X\,\partial_w\Lambda\,\varphi \,dw\,dt.
	\]
	Finally, using again the strong convergence of $\tX_N$ and the weak-* convergence
	of $\tLambda_N$, we get
	\[
	T_N^3 \to - \int_0^T \int_0^1 X\,\Lambda\,\partial_w\varphi \,dw\,dt.
	\]
	Passing to the limit in the above identity yields
	\[
	- \int_0^T \int_0^1 \Lambda\,\varphi \,dw\,dt
	- \int_0^T \int_0^1 X\,\partial_w\Lambda\,\varphi \,dw\,dt
	- \int_0^T \int_0^1 X\,\Lambda\,\partial_w\varphi \,dw\,dt
	= 0.
	\]
	Equivalently,
	\[
	\int_0^T \int_0^1 (\partial_w X -1)\Lambda\,\varphi \,dw\,dt = 0,
	\]
	for every $\varphi \in \mathcal C_c^1((0,T)\times(0,1))$. Therefore,
	\[
	(\partial_w X -1)\Lambda = 0
	\qquad \text{in } \mathcal D'((0,T)\times(0,1)).
	\]
\end{itemize}
\end{proof}

We have obtained the convergence of the interpolated variables and identified the limit $(X,\Lambda)$ as a solution of the first-order differential inclusion.
We now return to the second-order formulation of the dynamics. 
In particular, we introduce the limit velocity and pressure, and recover the corresponding second-order system in Lagrangian coordinates.

\subsection{The second-order differential inclusion}\label{ssec:cvg2}

We now turn to the second-order formulation of the dynamics.
We proceed as follows. 
First, we derive uniform bounds on the interpolated velocities and obtain compactness. Then, we introduce the limit pressure as a measure and establish its convergence. Finally, we pass to the limit in the second-order differential inclusion and in particular in the exclusion relation.

\begin{prop}[Controls on $U_N$ and $\tU_N$]
	From the microscopic estimates, we obtain
	\begin{equation}\label{eq:bound-UN-Linf}
	\|U_N\|_{L^\infty((0,T)\times(0,1))} + \|\tU_N\|_{L^\infty((0,T)\times(0,1))} \leq C.
	\end{equation}
	as well as the Oleinik-type inequality
	\begin{equation}\label{eq:Oleinik-tUN}
	\partial_w \tU_N < \dfrac{\partial_w \tX_N}{t}.
	\end{equation}
	As a consequence, we have a BV bound away from the initial time: for every $\alpha > 0$
	\begin{equation}\label{pwUN-L1}
	\sup_{t \in[\alpha,T]} \|\partial_w \tU_N(t)\|_{L^1(0,1)} \leq C_\alpha,
	\end{equation}
	for some constant $C_\alpha>0$ independent of $N$.
\end{prop}

\medskip
\begin{proof}
		The $L^\infty$ bound~\eqref{eq:bound-UN-Linf} follows directly from the
		microscopic estimates.
		Estimate~\eqref{eq:Oleinik-tUN} is the continuous counterpart of
		\eqref{eq:Oleinik-micro}.\\
		We now derive the $L^1$ bound on $\partial_w \tU_N$. We write
		\begin{align*}
		\|\partial_w \tU_N(t)\|_{L^1(0,1)}
		&= 2 \int_0^1 (\partial_w \tU_N(t))_+\, dw 
		- \int_0^1 \partial_w \tU_N(t)\, dw.
		\end{align*}
		Using the Oleinik estimate~\eqref{eq:Oleinik-tUN}, we obtain
		\[
		(\partial_w \tU_N)_+ \le \frac{\partial_w \tX_N}{t},
		\]
		hence
		\[
		\int_0^1 (\partial_w \tU_N(t))_+\, dw
		\le \frac{1}{t} \int_0^1 \partial_w \tX_N(t,w)\,dw
		= \frac{1}{t} (x_N(t) - x_0(t)).
		\]
		Moreover,
		\[
		\int_0^1 \partial_w \tU_N(t)\, dw
		= \tU_N(t,1) - \tU_N(t,0).
		\]
		Combining these estimates, we obtain
		\[
		\|\partial_w \tU_N(t)\|_{L^1(0,1)}
		\le \frac{2}{t}(x_N(t)-x_0(t)) 
		- \big(\tU_N(t,1)-\tU_N(t,0)\big),
		\]
		which is uniformly bounded with respect to $N$ thanks to the estimates on $X_N$ and $\tU_N$.
\end{proof}

\medskip

We now introduce the pressure variable associated with the second-order formulation. 
At the discrete level, it is defined as the time derivative of the Lagrange multipliers.
Due to the low regularity in time of the multipliers, the pressure is only defined as a distribution. 
We therefore establish a uniform bound on the pressure as a measure in time.

\begin{prop}[Bound on the pressure]\label{prop:pressure-bound}
	The interpolated pressure $\tP_N$ satisfies the uniform estimate
	\begin{equation}
	\|\tP_N\|_{\mathcal{M}(0,T; L^1(0,1))} \leq C,
	\end{equation}
	where $C>0$ is independent of $N$. 
\end{prop}

\medskip

\begin{proof}
	We test Equation~\eqref{eq:PDE-Lagr-N-2} against the test function
	\begin{align}
	\varphi(t,w) 
	& = \Big((x_N(t) - x_0(t))w -  \tX_N(t,w)\Big)\, \phi(t), \qquad \phi \in \mathcal{C}^1_c((0,T)), \ \phi \geq 0,
	\end{align}
	Observe that at the boundaries $\varphi(t,0)=\varphi(t,1)=0$ and	
	\begin{align*}
	\partial_t \varphi(t,w) & = \Big((u_N(t) - u_0(t))w - \widetilde U_N(t,w)\big)\phi(t) +  \Big((x_N(t) - x_0(t))w- \tX_N(t,w) \Big)\, \phi'(t),\\
	\partial_w \varphi(t,w) & = \Big((x_N(t) - x_0(t)) - \partial_w \tX_N(t,w)\Big)\, \phi(t).
	\end{align*}
	Using~\eqref{eq:PDE-Lagr-N-2}, we obtain
	\begin{align*}
	& \int_0^T \int_0^1 \Big( (x_N(t) - x_0(t)) -\partial_w \tX_N(t,w)\Big)\phi(t) \tP_N(w,dt)dw \\
	& = - \int_0^T \int_0^1 \left( (U_N(t,0) -U_N(t,1))w - \widetilde U_N(t,w)) \right)U_N(t,w) \phi(t)\,  dwdt \\
	& \quad- \int_0^T \int_0^1 U_N(t,w) \Big((x_N(t) - x_0(t))w -  \tX_N(t,w)\Big)\, \phi'(t)\,  dwdt.
	\end{align*}
	By the previously established bounds on $U_N$ and $\tX_N$, the right-hand side is uniformly bounded. 
	Hence there exists a constant $C>0$, independent of $N$, such that
	\begin{align}\label{eq:test-pressure-bound}
	& \left|\int_0^T\int_0^1  \Big((x_N(t) - x_0(t)) -  \partial_w \tX_N(t,w)\Big)\phi(t) \tP_N(dt,w)dw \right|\\
	& \leq C(\|U_N\|_{L^\infty_{t,w}}^2 + \|U_N\|_{L^2_tL^2_w}^2 + \|\tX_N\|_{L^2_t L^2_w}^2).
	\end{align}
	We now use the exclusion relation. Since
	\[
	(\partial_w \tX_N-1)\,\tP_N=0,
	\]
	we may replace $\partial_w\tX_N$ by $1$ on the support of $\tP_N$, and therefore
	\begin{align*}
	&\int_0^T\int_0^1
	\Big((x_N(t) - x_0(t))-\partial_w \tX_N(t,w)\Big)\phi(t)\,\tP_N(dt,w)\,dw
	\\
	&=
	\int_0^T\int_0^1
	\Big((x_N(t)-x_0(t)) -1\Big)\phi(t)\,\tP_N(dt,w)\,dw.
	\end{align*}
	Next, using
	\[
	x_N(t)-x_0(t)
	= (x_N(t)-x_1(t)) + (x_1(t)-x_0(t)),
	\]
	together with
	\[
	x_N(t)-x_1(t)\geq 2r(N-1), \qquad x_1(t)-x_0(t) = 2r+\delta,
	\]
	and recalling that $2r=1/N$, we obtain
	\[
	(x_N(t)-x_0(t)) -1 \geq 2r(N-1)+(2r+\delta) -1 =\delta.
	\]
	Hence
	\begin{align*}
	&\int_0^T\int_0^1
	\Big((x_N(t)-x_0(t)) - \partial_w \tX_N(t,w)\Big)\phi(t)\,\tP_N(dt,w)\,dw
	\\
	&\geq \delta	\int_0^T\int_0^1 \phi(t)\,\tP_N(dt,w)\,dw.
	\end{align*}
	Combining this with~\eqref{eq:test-pressure-bound}, we conclude that
	\[
	\int_0^T\int_0^1 \phi(t)\,\tP_N(dt,w)\,dw \le \frac{C}{\delta},
	\]
	for every nonnegative $\phi\in\mathcal C_c^1((0,T))$. This proves the desired
	uniform bound on $\tP_N$. 

\end{proof}

\medskip

\begin{rmk}
	In the previous proof, the parameter $\delta$, defined as the distance between the first particle and the fictitious particle $0$ introduced in~\eqref{df:x0}, plays a crucial role in deriving uniform estimates on the pressure. 
	This technical device can be interpreted as the microscopic counterpart of assumptions commonly imposed at the macroscopic level, where fully congested configurations are excluded, we refer for instance to Eq.~(13) in~\cite{perrin2015}. 
	In this sense, the introduction of $\delta$ prevents the formation of a globally saturated configuration.
\end{rmk}

\medskip
We are now in a position to pass to the limit in the second-order differential inclusion. 
Combining the uniform bounds on the velocities and the pressure, we extract converging subsequences and identify the limit system.

\medskip

\begin{prop}[Limit of the second-order differential inclusion]\label{prop:2ndorder-LagrPDE}
There exist
\[
U \in L^\infty(0,T;BV(0,1)),
\qquad
P \in \mathcal{M}_+((0,T)\times(0,1)),
\]
such that, up to extraction of a subsequence, the following convergences hold:
\begin{align*}
U_N &\rightharpoonup^* U \quad \text{in} \quad L^\infty((0,T) \times(0,1)), \\
\tU_N &\rightharpoonup^* U \quad \text{in} \quad L^\infty(\alpha,T;BV(0,1)), ~ \forall \ \alpha > 0,\\
\tP_N &\rightharpoonup^* P \quad \text{in}~ \mathcal{M}_+((0,T) \times (0,1)).
\end{align*}
Moreover, the limit $(X,U,P)$ satisfies
\begin{subnumcases}{\label{eq:PDE-Lagr-2}}
\partial_t X = U, \qquad \text{a.e},\\
\partial_t U + \partial_w P = 0 \qquad \text{in} \quad \mathcal{D}',\\
\partial_w X \geq 1, \quad (\partial_w X - 1)\ P = 0, \quad P \geq 0 \quad  \text{in} \quad \mathcal{D}',
\end{subnumcases}
with the Oleinik-type inequality
\begin{equation}\label{eq:Oleinik-Lagr}
\partial_w U < \dfrac{\partial_w X}{t} \quad \text{in }~ \mathcal{D'}((0,T) \times (0,1)).
\end{equation}
\end{prop}

\medskip
\begin{rmk}
	At the discrete level, the boundary pressures vanish since $p_0^N=p_N^N=0$, as a consequence of	$\lambda_0^N=\lambda_N^N=0$. 
	At the limit, however, the pressure $P$ is only obtained as a measure on $(0,T)\times(0,1)$, so that no classical spatial trace at $w=0,1$ is available in general.
	The natural boundary condition is therefore encoded at the level of the primitive:
	\[
	P = \partial_t\Lambda \qquad\text{with}\qquad
	\Lambda(t,0) = \Lambda(t,1) = 0 \quad\text{for a.e. } t \in (0,T).
	\]
	In particular, the boundary values of $P$ vanish only in the weak sense obtained by differentiating	the boundary condition on $\Lambda$ with respect to time.
\end{rmk}

\medskip
\begin{proof}
\begin{itemize}
	\item Thanks to the bounds~\eqref{eq:bound-UN-Linf} and~\eqref{pwUN-L1}, we have directly
	\begin{align*}
	U_N &\to U \quad \text{weakly-* in} \quad L^\infty((0,T) \times(0,1)), \\
	\tU_N &\to \tU \quad  \text{weakly-* in} \quad L^\infty(\alpha,T;BV(0,1)) \quad \forall \ \alpha > 0.
	\end{align*}
	We then show that the two limits $U$ and $\tU$ coincide. Indeed,
	\[
	\int_0^1 \left|\tU_N(t,w) - U_N(t,w)\right|dw
	= \dfrac{1}{2N} \sum_{i=1}^N |u_i(t) - u_{i-1}(t)|
	= \dfrac{1}{2N} |u_N(t) - u_0(t)| \leq \dfrac{C}{N} \to 0.
	\]
	Hence $\tU_N-U_N\to0$ strongly in $L^\infty(0,T;L^1(0,1))$, so both sequences converge to the same limit $U$.
	\item Since the sequence $(\tP_N)$ is bounded in $\mathcal{M}(0,T; L^1(0,1))$, there exists a limit measure $P$ such that, up to a subsequence, $(\tP_N)$ converges to $P$ in $\mathcal{M}_+((0,T) \times (0,1))$.
	\item	Let $\Psi \in \mathcal{C}^1_c(0,T; L^1(0,1))$. We have
	\[
	\langle P_N - \tP_N, \Psi \rangle
	= - \int_0^T \int_0^1 \big(\Lambda_N(t,w) - \tLambda_N(t,w)\big) \partial_t \Psi(t,w) \, dw dt.
	\]
	Hence, using~\eqref{eq:diffLambda_N}
	\[
	\left|\langle P_N - \tP_N, \Psi \rangle  \right|
	\leq \dfrac{C}{N} \|\partial_t \Psi\|_{L^1(0,T; L^\infty(0,1))},
	\]
	and therefore
	\begin{equation*}
	P_N - \tP_N \to 0 \quad \text{in}\quad W^{-1,\infty}(0,T; L^1(0,1)).
	\end{equation*}
	\item On the other hand, we have $(\partial_w \tX_N)_N$ bounded in $W^{1,\infty}(\alpha,T; L^1(0,1))$ (recall that $\tU_N = \partial_t \tX_N$ satisfies the estimate~\eqref{pwUN-L1}). As a consequence, we have
	\[
	\big(\partial_w \tX_N - 1\big) P_N = \big(\partial_w \tX_N - 1\big) \tP_N + R^1_N, 
	\]
	with $R^1_N :=  \big(\partial_w \tX_N - 1\big) (P_N - \tP_N)$ which tends to $0$ in the sense of distributions thanks to the previous item.

	\item Now, for $\Psi \in \mathcal C_c^1((0,T)\times(0,1))$, we have
	\begin{align*}
	0
	= \langle \tP_N,(\partial_w\tX_N-1)\Psi\rangle
	&= -\langle \partial_w\tP_N,(\tX_N-w)\Psi\rangle
	-\langle \tP_N,(\tX_N-w)\partial_w\Psi\rangle \\
	&=: T_N^1 + T_N^2.
	\end{align*}
	
	Since $\tX_N \to X$ strongly in $\mathcal C([0,T]\times[0,1])$ and
	$\tP_N \rightharpoonup^\ast P$ in $\mathcal M((0,T)\times(0,1))$, we obtain
	\[
	-T_N^2
	= \langle \tP_N,(\tX_N-w)\partial_w\Psi\rangle
	\longrightarrow
	\langle P,(X-w)\partial_w\Psi\rangle.
	\]
	
	On the other hand, we have
	\[
	\partial_w \tP_N = - \partial_t U_N
	\qquad \text{in } \mathcal D'((0,T)\times(0,1)).
	\]
	Since $U_N \rightharpoonup^* U$ in $L^\infty$, we deduce that
	\[
	\partial_t U_N \rightharpoonup^* \partial_t U
	\quad \text{in } W^{-1,\infty}(0,T;L^\infty(0,1)).
	\]
	Therefore, up to identification of the limit,
	\[
	\partial_w \tP_N \rightharpoonup^* -\partial_t U
	\qquad \text{in } W^{-1,\infty}(0,T;L^\infty(0,1)),
	\]
	that is,
	\[
	\partial_w P = -\partial_t U
	\qquad \text{in } \mathcal D'((0,T)\times(0,1)).
	\]
	Combining this weak convergence with the strong convergence of $\tX_N$, we obtain
	\begin{equation*}
	- T^1_N
	= \langle \partial_w \tP_N , (\tX_N- w) \Psi \rangle
	\ \longrightarrow \  \langle  \partial_w P , (X- w) \Psi \rangle.
	\end{equation*}
	We conclude that 
	\[
	\langle  \partial_w P , (\tX- w) \,\Psi \rangle + \langle  P , (\tX- w) \partial_w\Psi \rangle  = 0.
	\]
	This is exactly the weak formulation of the relation exclusion
	\[
	(\partial_w X-1)\,P=0
	\quad
	\text{in } \mathcal D'((0,T)\times(0,1)).
	\]
\end{itemize}
\end{proof}

\subsection{Eulerian variables and the constrained Euler equations}

\medskip

We have identified the limit system in Lagrangian coordinates, described by $(X,U,P)$. 
We now reformulate this system in Eulerian variables by introducing the density, velocity and pressure in physical space.
This change of variables allows us to identify the corresponding system of constrained Euler equations~\eqref{eq:macro-euler}.

\medskip
A key preliminary step is to ensure that the Eulerian density is well-defined and enjoys sufficient regularity. 
This is the purpose of the next lemma.

\begin{lem}[Regularity of the Eulerian density]\label{lem:density}
	Let $X \in W^{1,\infty}((0,T)\times(0,1))$ be such that, for a.e. $t\in(0,T)$,
	the map $w \mapsto X(t,w)$ is nondecreasing and satisfies
	\[
	\partial_w X(t,w)\ge 1
	\qquad \text{for a.e. } w\in(0,1).
	\]
	Define
	\begin{equation}\label{df:rho}
	\rho(t,\cdot):=X(t,\cdot)_\# \mathcal L^1_{(0,1)}.
	\end{equation}
	Then the following properties hold:
	
	\begin{enumerate}
		\item The curve $t\mapsto \rho(t)$ belongs to $AC([0,T];\mathcal P_2(\mathbb R))$.
		
		\item There exists a bounded Borel function, still denoted by $\rho$, such that
		\[
		\rho(dt,dx)=\rho(t,x)\,dt\,dx,
		\qquad 0\le \rho(t,x)\le 1
		\quad \text{for a.e. } (t,x)\in(0,T)\times\mathbb R.
		\]
	\end{enumerate}
\end{lem}

\medskip
For sake of clarity, the proof of this lemma is postponed to the appendix, see~\ref{lem:density-app}.

The main result of this section is the following proposition which shows that a Lagrangian solution $(X,U,P)$ defines a weak solution $(\rho, \rho u,p)$ of the constrained Euler equations.

\medskip
\begin{prop}[Eulerian formulation]
	Let $(X,U,P)$ be the triplet constructed in Proposition~\ref{prop:2ndorder-LagrPDE}, solution to the second-order differential inclusion~\eqref{eq:PDE-Lagr-2} with initial data $(X^0,U^0)$ , and let $\Lambda$ be the associated multiplier, i.e.
	\[
	P = \partial_t \Lambda \quad \text{in } \mathcal D'((0,T)\times(0,1)).
	\]
	We define the Eulerian density $\rho$ by the previous Lemma~\ref{lem:density} (see the push-forward formula~\eqref{df:rho}). 
	Similarly, we define the Eulerian momentum and pressure by
	\begin{align*}
	(\rho u)(t,\cdot) &= X(t,\cdot)_\# \big(U(t,\cdot)\,\mathcal L^1_{(0,1)}\big),\\
	p &= (\mathrm{Id},X)_\# P.
	\end{align*}
	
	Then $(\rho,\rho u,p)$ is a distributional solution of the constrained Euler system~\eqref{eq:macro-euler} satisfying an Oleinik inequality $\partial_x u < 1/t$ in the sense of distributions.
\end{prop}

\medskip
\begin{rmk}
	One may equivalently write $u(t,x)=U(t,W(t,x))$, $p = \partial_t \Lambda(t,W(t,x))$ where $W$ is the generalized inverse of $X$, but this representation is not needed in the weak formulation.
\end{rmk}

\medskip
\begin{proof}
	Let $\phi \in \mathcal C_c^\infty([0,T)\times\mathbb{R})$.
	
	\medskip
	
	\noindent
	\textit{Mass equation.}
	Using the definition of $\rho$ as a push-forward, we write
	\begin{align*}
	\int_\mathbb{R} \phi(0,x)\,\rho^0(dx)
	&= - \int_0^T \frac{d}{dt} \int_\mathbb{R} \phi(t,x)\,\rho(t,x)\,dx\,dt \\
	&= - \int_0^T \frac{d}{dt} \int_0^1 \phi(t,X(t,w))\,dw\,dt.
	\end{align*}
	Differentiating under the integral sign, we obtain
	\begin{align*}
	\int_\mathbb{R} \phi(0,x)\,\rho^0(x)\, dx
	&= - \int_0^T \int_0^1 \big(
	\partial_t \phi(t,X(t,w))
	+ U(t,w)\,\partial_x \phi(t,X(t,w))
	\big)\,dw\,dt \\
	&= - \int_0^T \int_\mathbb{R}
	\big(\partial_t \phi + u\,\partial_x \phi\big)(t,x)\,\rho(t,dx)\,dt.
	\end{align*}
	This proves the mass equation in the weak sense.
	
	\medskip
	
	\noindent
	\textit{Momentum equation.}
	Similarly,
	\begin{align*}
	\int_\mathbb{R} \phi(0,x)\,u^0(x)\,\rho^0(x)\, dx
	&= \int_0^1 \phi(0,X(0,w))\,U^0(w)\,dw \\
	&= - \int_0^T \frac{d}{dt} \int_0^1 \phi(t,X(t,w))\,U^0(w)\,dw\,dt.
	\end{align*}
	Proceeding as before,
	\begin{align*}
	\int_\mathbb{R} \phi(0,x)\,u^0(x)\,\rho^0(x)\,dx
	& = - \int_0^T \int_0^1	\big(\partial_t \phi + U\,\partial_x \phi\big)(t,X(t,w))\,U^0(w)\,dw\,dt.
	\end{align*}
	Recall that $(X,\Lambda)$ satisfies the first-order Lagrangian system~\eqref{eq:Lagr-1},
	so that
	\[
	U = \partial_t X = U^0 - \partial_w \Lambda.
	\]
	we obtain
	\begin{align*}
	\int_\mathbb{R} \phi(0,x)\,u^0(x)\,\rho^0(x)\, dx
	&= - \int_0^T \int_0^1 \big(\partial_t \phi + U\,\partial_x \phi\big)(t,X(t,w))\,U(t,w)\,dw\,dt \\
	&\quad	- \int_0^T \int_0^1	\big(\partial_t \phi + U\,\partial_x \phi\big)(t,X(t,w))\,\partial_w \Lambda(t,w)\,dw\,dt.
	\end{align*}
	The first term gives
	\[
	- \int_0^T \int_\mathbb{R}
	\big(\partial_t \phi(t,x) + u(t,x)\,\partial_x \phi(t,x)\big)\,u(t,x)\,\rho(t,x) \ dxdt.
	\]
	For the second term, we have
	\begin{align*}
	& -\int_0^T \int_0^1	\big(\partial_t \phi + U\,\partial_x \phi\big)(t,X(t,w))\,\partial_w \Lambda(t,w)\ dw dt \\
	& = -\int_0^T \int_0^1	\dfrac{d}{dt} \big(\phi(t,X(t,w))\big) \,\partial_w \Lambda(t,w)\ dw dt \\
	& =  - \langle \partial_x \phi , p \rangle_{\mathcal{D}'_{t,x},\mathcal{D}_{t,x}}.
	\end{align*}
    using the definition of $p$.
	This proves the momentum equation.
	
	\medskip

	\noindent
	\textit{Complementary relation.}~
	By Lemma~\ref{lem:density}, the Eulerian density $\rho$ admits a bounded Borel representative, still denoted by $\rho$, satisfying
	\[
	0 \leq \rho \leq 1	\qquad \text{a.e. on } (0,T)\times\mathbb{R}.
	\]
	In general, the product of a bounded function with a Radon measure depends on the choice of the representative. 
	In the present framework, the Lagrangian structure allows us to select a canonical representative of $\rho$.
	Indeed, recalling that
	\[
	\rho(t,\cdot) = X(t,\cdot)_{\#}\mathcal L^1_{(0,1)},
	\]
	we can choose $\rho$ so that
	\[
	\rho(t,X(t,w)) = \frac{1}{\partial_w X(t,w)}
	\qquad \text{for a.e. } (t,w)\in (0,T)\times(0,1).
	\]
	In particular, on the congested set
	\[
	\Omega := \{(t,w)\in (0,T)\times(0,1)\ ; \ \partial_w X(t,w)=1\},
	\]
	we have
	\[
	\rho(t,X(t,w)) = 1
	\qquad \text{for a.e. } (t,w)\in \Omega.
	\]
	On the other hand, by construction, the Lagrangian pressure $P$ is supported on $\Omega$, and the Eulerian pressure is defined by
	\[
	p = (\mathrm{Id},X)_{\#} P.
	\]
	With this choice of representative, the product $(1-\rho)p$ is well defined	as a Radon measure, and for any $\phi \in C_c^\infty((0,T)\times\mathbb{R})$, we have
	\begin{align*}
	\langle (1-\rho)p,\phi\rangle_{\mathcal{D}'_{t,x},\mathcal{D}_{t,x}}
	& = \int_{(0,T)\times\mathbb R} (1-\rho(t,x))\,\phi(t,x)\,p(dt,dx) \\
	& = \int_{(0,T)\times(0,1)}	(1-\rho(t,X(t,w)))\,\phi(t,X(t,w))\,P(dt,dw).
	\end{align*}
	Since $P$ is supported on $\Omega$ and $\rho(t,X(t,w))=1$ on $\Omega$, the integrand vanishes $P$-almost everywhere, and therefore
	\[
	\langle (1-\rho)p,\phi\rangle_{\mathcal{D}'_{t,x},\mathcal{D}_{t,x}} = 0.
	\]
	We conclude that
	\[
	(1-\rho)p = 0
	\quad \text{in} \quad \mathcal{D}'((0,T)\times\mathbb R).
	\]

	\medskip
	
	\noindent
	\textit{Oleinik inequality.}
	We use the Lagrangian Oleinik inequality~\eqref{eq:Oleinik-Lagr}, which states that for any
	$\phi \in \mathcal C_c^\infty((0,T)\times(0,1))$, $\phi \ge 0$,
	\[
	-\int_0^T \int_0^1 U(t,w)\,\partial_w \phi(t,w)\,dw\,dt
	< \int_0^T \int_0^1 \frac{1}{t}\,\phi(t,w)\,\partial_w X(t,w)\,dw\,dt.
	\]
	Let $\varphi \in \mathcal C_c^\infty((0,T)\times\mathbb R)$, $\varphi \ge 0$, and define
	\[
	\phi(t,w) := \varphi(t,X(t,w)).
	\]
	Then
	\[
	\partial_w \phi(t,w)
	= \partial_x \varphi(t,X(t,w))\,\partial_w X(t,w).
	\]
	Substituting into the previous inequality yields
	\begin{align*}
	- \int_0^T \int_0^1 U(t,w)\,\partial_x \varphi(t,X(t,w))\,\partial_w X(t,w)\,dw dt
	< \int_0^T \int_0^1 \frac{1}{t}\,\varphi(t,X(t,w))\,\partial_w X(t,w)\,dw dt.
	\end{align*}
	Using the change of variables induced by $X$, we obtain
	\begin{align*}
	- \int_0^T \int_{\mathbb R} u(t,x)\,\partial_x \varphi(t,x)\,dx\,dt
	< \int_0^T \int_{\mathbb R} \frac{1}{t}\,\varphi(t,x)\,dx\,dt.
	\end{align*}
	This proves the Eulerian Oleinik inequality.
\end{proof}

\appendix
\section{Karush-Kuhn-Tucker theory}\label{app:KKT}

In this appendix, we recall a classical result from convex optimization,
which provides a direct characterization of constrained minimizers via
Lagrange multipliers.

\begin{thm}[Karush--Kuhn--Tucker conditions, \cite{boyd2004} Chapter 5]
	Let $f:\mathbb{R}^N \to \mathbb{R}$ be a convex and continuously differentiable function, and let $g_i:\mathbb{R}^N \to \mathbb{R}$, $i=1,\dots,m$, be convex functions.
	Consider the constrained minimization problem
	\begin{equation}\label{eq:KKT-problem}
	\min_{\bx \in \mathbb{R}^N} f(x)
	\quad \text{subject to} \quad
	g_i(x) \le 0 \quad \text{for } i=1,\dots,m.
	\end{equation}
	Assume that the feasible set
	\[
	\mathcal{K} := \{\bx \in \mathbb{R}^N ; \ g_i(\bx) \le 0 ~ \ \forall \ i\}
	\]
	is nonempty and that the Slater condition holds, i.e. there exists $\bar \bx \in \mathcal{K}$ such that
	\[
	g_i(\bar \bx) < 0 \quad \text{for all } i.
	\]
	Then $\bx^\star \in K$ is a solution to~\eqref{eq:KKT-problem} if and only if there exist multipliers $\lambda_i \geq 0$, $i=1,\dots,m$, such that
	\begin{subequations}\label{eq:KKT}
		\begin{align}
		\nabla f(\bx^\star) + \sum_{i=1}^m \lambda_i \nabla g_i(\bx^\star) & = 0, \label{eq:KKT-stationarity}\\
		g_i(\bx^\star) &\leq 0, \label{eq:KKT-primal}\\
		\lambda_i &\geq 0, \label{eq:KKT-dual}\\
		\lambda_i\, g_i(\bx^\star) &= 0 \quad \text{for all } i. \label{eq:KKT-complementarity}
		\end{align}
	\end{subequations}
\end{thm}

\medskip

We apply this result to the projection problem defining the microscopic dynamics.
For a given $t \geq 0$, consider
\[
\min_{\bx \in \mathcal{K}} \frac{1}{2} \big\|\bx - (\bx^0 + t \bu^0)\big\|^2,
\]
where
\[
\mathcal{K} = K^N = \{\bx \in \mathbb{R}^N : x_{i+1} - x_i \geq 2r, \ i=1,\dots,N-1\}.
\]
This problem is convex with affine constraints. 
In our setting, the Slater condition is satisfied. Indeed, one may consider for instance
\[
\bar x_i := i\,(2r+\delta), \qquad i=1,\dots,N,
\]
for some $\delta>0$, so that
\[
\bar x_{i+1} - \bar x_i = 2r+\delta > 2r.
\]
Therefore, the KKT conditions apply and yield the existence of multipliers $\lambda_i(t) \geq 0$ such that
\begin{align*}
\bx(t) - (\bx^0 + t \bu^0) +  N \sum_{i=1}^{N-1} \lambda_i(t) (e_i - e_{i+1}) 
& = 0,\\
\lambda_i(t)\,\big(x_{i+1}(t) - x_i(t) - 2r\big) 
& = 0.
\end{align*}
This provides directly the multipliers and the Signorini condition introduced in the main text.

\section{Technical lemmas}

\begin{lem}[Regularity of the Eulerian density]\label{lem:density-app}
	Let $X \in W^{1,\infty}((0,T)\times(0,1))$ be such that, for a.e. $t\in(0,T)$,
	the map $w \mapsto X(t,w)$ is nondecreasing and satisfies
	\[
	\partial_w X(t,w)\ge 1
	\qquad \text{for a.e. } w\in(0,1).
	\]
	Define
	\begin{equation}\label{df:rho-app}
	\rho(t,\cdot):=X(t,\cdot)_\# \mathcal L^1_{(0,1)}.
	\end{equation}
	Then the following properties hold:
	
	\begin{enumerate}
		\item The curve $t\mapsto \rho(t)$ belongs to $AC([0,T];\mathcal P_2(\mathbb R))$.
		
		\item There exists a bounded Borel function, still denoted by $\rho$, such that
		\[
		\rho(dt,dx)=\rho(t,x)\,dt\,dx,
		\qquad 0\le \rho(t,x)\le 1
		\quad \text{for a.e. } (t,x)\in(0,T)\times\mathbb R.
		\]
	\end{enumerate}
\end{lem}

\begin{proof}
	\begin{itemize}
		\item {\it Absolute continuity in time with values in $\mathcal P_2(\mathbb R)$.}
		Let $0\le s<t\le T$. Consider the transport plan
		\[
		\gamma_{s,t}:=(X(s,\cdot),X(t,\cdot))_\# \mathcal L^1_{(0,1)}.
		\]
		By construction, $\gamma_{s,t}$ is a coupling between $\rho(s)$ and $\rho(t)$ (see for instance~\cite{santambrogio2015}).
		Therefore,
		\[
		W_2(\rho(s),\rho(t))^2
		\le \int_0^1 |X(t,w)-X(s,w)|^2\,dw.
		\]
		Since $\partial_t X \in L^\infty((0,T)\times(0,1))$, we have for a.e. $w\in(0,1)$
		\[
		X(t,w)-X(s,w)=\int_s^t \partial_t X(\tau,w)\,d\tau.
		\]
		Hence,
		\[
		\|X(t,\cdot)-X(s,\cdot)\|_{L^2(0,1)}
		\le \int_s^t \|\partial_t X(\tau,\cdot)\|_{L^2(0,1)}\,d\tau,
		\]
		and thus
		\[
		W_2(\rho(s),\rho(t))
		\le \int_s^t \|\partial_t X(\tau,\cdot)\|_{L^2(0,1)}\,d\tau.
		\]
		This proves that $t\mapsto \rho(t)$ belongs to $AC([0,T];\mathcal P_2(\mathbb R))$.
		
		\medskip
		
		\item \textit{Existence of a bounded Borel representative.}
		For a.e. fixed $t\in(0,T)$, define
		\[
		\mu_t:=X(t,\cdot)_\#\mathcal L^1_{(0,1)}=\rho(t,\cdot).
		\]
		We show that $\mu_t \ll \mathcal L^1_{\mathbb R}$ and that its density is bounded
		by $1$.
		
		Let $I=(a,b)\subset \mathbb R$ be an open interval. Since $X(t,\cdot)$ is nondecreasing,
		the preimage $X(t,\cdot)^{-1}(I)$ is an interval (possibly empty), say $(\alpha,\beta)\subset(0,1)$.
		Using the absolute continuity of $X(t,\cdot)$ and the bound $\partial_w X(t,\cdot)\ge 1$, we get
		\[
		b-a \ge X(t,\beta)-X(t,\alpha)
		= \int_\alpha^\beta \partial_w X(t,w)\,dw
		\ge \beta-\alpha.
		\]
		Therefore,
		\[
		\mu_t(I)=\mathcal L^1\bigl(X(t,\cdot)^{-1}(I)\bigr)\le |I|.
		\]
		By approximation, the same inequality holds for every Borel set $A\subset\mathbb R$:
		\[
		\mu_t(A)\le |A|.
		\]
		Hence $\mu_t \ll \mathcal L^1_{\mathbb R}$ and its Radon--Nikodym density,
		still denoted by $\rho(t,\cdot)$, satisfies
		\[
		0\le \rho(t,\cdot)\le 1
		\qquad \text{for a.e. } t\in(0,T).
		\]
		We now pass to space-time. Define
		\[
		\mu := (\mathrm{Id},X)_\#\bigl(\mathcal L^1_{(0,T)}\otimes \mathcal L^1_{(0,1)}\bigr)
		\quad \text{on } (0,T)\times\mathbb R.
		\]
		Then for every Borel set $B\subset (0,T)\times\mathbb R$,
		\[
		\mu(B)=\int_0^T \mu_t(B_t)\,dt,
		\qquad
		B_t:=\{x\in\mathbb R:\ (t,x)\in B\}.
		\]
		Using the previous estimate slice by slice, we find
		\[
		\mu(B)\le \int_0^T |B_t|\,dt = |B|.
		\]
		Thus $\mu \ll \mathcal L^1_{(0,T)\times\mathbb R}$, and by the Radon--Nikodym theorem
		there exists a bounded Borel function, still denoted by $\rho$, such that
		\[
		\mu(dt,dx)=\rho(t,x)\,dt\,dx,
		\qquad 0\le \rho(t,x)\le 1
		\quad \text{for a.e. } (t,x)\in(0,T)\times\mathbb R.
		\]
	\end{itemize}
\end{proof}

\section{Uniqueness issues at the macroscopic level}\label{app:uniqueness}

\subsection{First-order Lagrangian system}

\begin{prop}[Monotone inclusion and uniqueness for \eqref{eq:Lagr-1}]
	Let $(X,\Lambda)$ be a solution to \eqref{eq:Lagr-1}, and define
	\[
	\xi := \partial_w \Lambda.
	\]
	Then, for a.e. $t\in(0,T)$, one has
	\[
	\xi(t)\in \partial I_K(X(t)),
	\qquad
	K := \{Y\in L^2(0,1);\  \partial_w Y \ge 1 ~ \text{in the sense of distributions} \}.
	\]
	Equivalently,
	\[
	\partial_t X(t) + \partial I_K(X(t)) \ni U^0.
	\]
	As a consequence, \eqref{eq:Lagr-1} admits at most one solution.
\end{prop}

\begin{proof}
	Let $Y\in K$. Using the boundary conditions $\Lambda(t,0)=\Lambda(t,1)=0$, we compute
	\[
	\int_0^1 \partial_w\Lambda(t,w)\,(Y(w)-X(t,w))\ dw
	= -\int_0^1 \Lambda(t,w)\,\partial_w(Y-X)(w)\ dw.
	\]
	Since $\Lambda \geq 0$, $\partial_w Y \geq 1$, and
	\[
	(\partial_w X - 1)\Lambda=0,
	\]
	we obtain
	\[
	-\Lambda\,\partial_w(Y-X)
	= \Lambda(\partial_w X-\partial_w Y)
	= \Lambda(1-\partial_w Y) 
	\leq 0.
	\]
	Therefore
	\[
	\int_0^1 \partial_w\Lambda(t,w)\,(Y(w)-X(t,w))\,dw \leq 0,
	\]
	which proves that $\partial_w\Lambda(t)\in \partial I_K(X(t))$.
	
	Now let $X^1,X^2$ be two solutions with the same initial datum $X^0$, and set $\xi^a:=\partial_w\Lambda^a\in \partial I_K(X^a)$, $a=1,2$.
	Subtracting the two equations yields
	\[
	\partial_t(X^1-X^2) + (\xi^1 - \xi^2) = 0 .
	\]
	Multiplying by $X^1 - X^2$ and integrating over $(0,1)$, we get
	\[
	\frac{1}{2} \frac{d}{dt} \|X^1-X^2\|_{L^2(0,1)}^2
	= - \int_0^1 \big(\xi^1 - \xi^2 \big) \big( X^1-X^2 \big) \ dw \leq 0,
	\]
	by the monotonicity of $\partial I_K$. 
	Since $X^1(0) = X^2(0) = X^0$, we conclude that	$X^1 = X^2$ a.e..
\end{proof}


\subsection{Non-uniqueness of the second-order dynamics}

In contrast with the first-order dynamics studied in the previous subsection,
which is well-posed thanks to its underlying monotone structure, the
second-order system~\eqref{eq:PDE-Lagr-2}
\[
\begin{cases}
\partial_t X = U,\\
\partial_t U + \partial_w P = 0,\\
\partial_w X \geq 1, \quad (\partial_w X-1)P=0,\quad P\geq0,
\end{cases}
\]
does not admit a unique solution in general. This lack of uniqueness persists
even if one strengthens the formulation by imposing additional regularity on
the pressure $P$ (so that boundary values $P_{|w=0}=P_{|w=1}=0$ are well-defined)
and by enforcing the Oleinik-type condition~\eqref{eq:Oleinik-Lagr}.
The key difference with the first-order formulation lies in the absence of an
explicit monotone structure or projection principle at the level of
\eqref{eq:PDE-Lagr-2}, which prevents any direct contraction argument.

\medskip
To illustrate this phenomenon, we consider the initial datum
\[
X^0(w) = w \mathbf{1}_{[0,1/2)}(w) + (w+ \eta)\mathbf{1}_{[1/2,1]}(w),
\qquad
U^0(w) = \mathbf{1}_{[0,1/2)} - \mathbf{1}_{[1/2,1]},
\]
which corresponds to two blocks moving towards each other and colliding
in a time of order $\eta>0$.
The free evolution before the collision is given by
\[
X^{\mathrm f}(t,w) = X^0(w) + t U^0(w).
\]
The collision occurs at time $t^* = \eta/2$, at which point one has
\[
X^{\mathrm f}(t^*,w) = w + \frac{\eta}{2}.
\]
After the collision time, several continuations are possible.

\begin{itemize}
	\item A first solution is given by
	\begin{align*}
	X^1(t,w) &= X^{\mathrm f}(t^*,w) = w + \frac{\eta}{2}, \\
	U^1(t,w) &= 0,
	\hspace{7cm} \forall \ t \geq t^*,\\
	P^1(t,w) &= \Big[w \mathbf{1}_{[0,1/2)}(w) - (1-w)\mathbf{1}_{[1/2,1]}(w)\Big]\,
	\delta_{t=t^*},
	\end{align*}
	which corresponds to a perfectly inelastic collision, where the two blocks
	stick together.
	
	\item A second solution exhibits a kind of rebound. For instance, we define
	\[
	U^2(t,w) = 2w-1, \qquad \forall\ t \geq t^*,
	\]
	and set
	\begin{align*}
	X^2(t,w) &= \Big(w + \frac{\eta}{2}\Big) + (t-t^*)U^2(t,w),
	\qquad \forall \ t \geq t^*,\\
	P^2(t,w) &= \Big[(2w-w^2)\mathbf{1}_{[0,1/2)}(w) - (w^2-1)\mathbf{1}_{[1/2,1]}(w)\Big]\,
	\delta_{t=t^*}.
	\end{align*}
	
	We check that the Oleinik condition is satisfied. Indeed,
	\[
	\partial_w U^2 = 2,
	\qquad
	\partial_w X^2(t,w) = 1 + 2(t-t^*),
	\]
	so that
	\[
	\frac{\partial_w X^2(t,w)}{t}
	= \frac{1}{t} + 2\frac{t-t^*}{t}.
	\]
	Hence, taking $\eta$ small (less that $1$)  so that $t^* = \eta/2 < 1/2$, one has
	\[
	\partial_w U^2 \leq \frac{\partial_w X^2}{t}
	\qquad \text{for all } t > t^*,
	\]
	and the Oleinik condition is satisfied for all times (the inequality being trivially satisfied for small times $t < t^*$).
\end{itemize}

\bigskip

The previous example shows that the second-order system \eqref{eq:PDE-Lagr-2} does not provide, by itself, a sufficient selection mechanism to ensure uniqueness.
A natural way to restore uniqueness is to supplement \eqref{eq:PDE-Lagr-2} with a projection principle on the velocity, inspired by the discrete sticky blocks dynamics. More precisely, given the monotone rearrangement $X(t)$, we introduce the closed subspace (introduced in~\cite{perrin2018} for the study of a toy model for granular flows)
\[
\H_{X(t)} := \Big\{ V \in L^2(0,1) \ ; \ V \text{ is constant on each connected component of } \{\partial_w X(t)=1\} \Big\},
\]
which is the macroscopic counterpart of~\eqref{df:H-micro} and encodes the fact that particles belonging to the same congested region must move with the same velocity.
We then postulate the projection condition
\begin{equation}\label{eq:velocity-projection}
U(t) = \mathbb P_{\H_{X(t)}}\big(U^0\big),
\end{equation}
where $\mathbb P_{\H_{X(t)}}$ denotes the orthogonal projection in $L^2(0,1)$.
This condition enforces the rigidity of congested regions and prevents the creation of non-physical velocity gradients inside clusters. 
The projection formula~\eqref{eq:velocity-projection} uniquely determines the velocity field $U(t)$ for each given configuration $X(t)$.
Since the evolution of $X$ is governed by $\partial_t X = U$, it follows that $X$ itself is uniquely determined. 
The pressure $P$ is uniquely recovered from the momentum equation and the boundary condition.
Indeed, this relation determines $\partial_w P$ uniquely as
\[
\partial_w P = -\partial_t U.
\]
Hence $P$ is determined up to an additive function of time. This indeterminacy is removed by imposing the boundary condition $P(t,0) = P(t,1) = 0$.
As a consequence, the whole triple $(X,U,P)$ is uniquely specified.\\
Therefore, the projection principle \eqref{eq:velocity-projection} provides a natural selection criterion that restores well-posedness of the second-order
dynamics, and can be seen as the macroscopic counterpart of the sticky blocks dynamics.
This also underlines the interest of the Lagrangian viewpoint.
In contrast with the Eulerian formulation, where no simple selection criterion is available, the Lagrangian description provides a direct and explicit mechanism through
the projection principle, leading to a canonical choice of solution.

\vspace{1cm}

\centerline{\bf Acknowledgements}
\vspace{2mm}
The work of the author is supported by the ANR BOURGEONS project, grant ANR-23-CE40-0014-01 of the French National Research Agency (ANR).

\bibliography{bibli}

\end{document}